\documentclass[11pt]{article}

\usepackage[dvips]{epsfig}
\usepackage{graphicx}
\usepackage{amssymb,graphics,amsmath,amsthm,amsopn,amstext,amsfonts}

\newcommand{\gap}{\vspace{0.1in}}

\newcommand{\wt}{\widetilde}

\newcommand{\wh}{\widehat}

\newcommand{\ol}{\overline}

\newcommand{\diam}{\diamond}
\newcommand{\Ical}{\mathcal I}
\newcommand{\Jcal}{\mathcal J}
\newcommand{\sgn}{{\rm sgn}}
\newcommand{\diag}{{\rm diag}}

\newtheorem{theorem}{Theorem}[section] 
\newtheorem{lemma}{Lemma}[section] 
\newtheorem{corollary}{Corollary}[section] 
\newtheorem{proposition}{Proposition}[section] 
\newtheorem{remark}{Remark}[section]

\begin{document}

\title{Least Sparsity of  $p$-norm based Optimization Problems with $p>1$}
%
%

\author{Jinglai Shen\footnote{Department of Mathematics and Statistics, University of Maryland Baltimore County, Baltimore, MD 21250, U.S.A. Emails: shenj@umbc.edu and smousav1@umbc.edu.} \ \ \ and \ \ \ Seyedahmad Mousavi  }

\maketitle

\begin{abstract}
Motivated by $\ell_p$-optimization arising from sparse optimization, high dimensional data analytics and statistics, this paper studies sparse properties of a wide range of  $p$-norm based  optimization problems with $p>1$, including generalized basis pursuit, basis pursuit denoising, ridge regression, and elastic net. It is well known that when $p>1$, these optimization problems lead to less sparse solutions. However, the quantitative characterization of the adverse sparse properties is not available. In this paper, by exploiting optimization and matrix analysis techniques, we give a systematic treatment of a broad class of $p$-norm based optimization problems for a general $p>1$ and show that optimal solutions to these problems attain full support, and thus have the least sparsity, for almost all measurement matrices and measurement vectors. Comparison to $\ell_p$-optimization with $0<p\le 1$ and implications to robustness are also given. These results  shed light on analysis and computation of general $p$-norm based optimization problems in various applications.
\end{abstract}

%
\section{Introduction}

Sparse optimization arises from various important applications of contemporary interest, e.g., compressed sensing, high dimensional data analytics and statistics, machine learning, and signal and image processing  \cite{CandesW_ISP08, FoucartRauhut_book2013, HastieTF_book09, RTibshirani_JRSS96}. The goal of sparse optimization is to recover the sparsest vector from observed data which are possibly subject to noise or errors, and it can be formulated as the $\ell_0$-optimization problem \cite{BeckE_SIOPT13, BKanzowS_SIOP16, FengMPSW_Techreport13}. Since the $\ell_0$-optimization problem is NP-hard,  it is a folklore in sparse optimization to use the $p$-norm or $p$-quasi-norm $\| \cdot \|_p$ with $p\in (0, 1]$ to approximate the $\ell_0$-norm to recover sparse signals \cite{Donoho_TIT06, Donoho_CPAM06}. Representative optimization problems involving such the $p$-norm include basis pursuit, basis pursuit denoising, LASSO, and elastic net; see Section~\ref{sect:formulation} for the details of these problems.
In particular, when $p=1$, it gives rise to a convex $\ell_1$-optimization problem which attains efficient numerical algorithms \cite{FoucartRauhut_book2013, WrightNF_TSP09}; when $0<p <1$, it yields a non-convex and non-Lipschitz optimization problem whose local optimal solutions can be effectively computed  \cite{ChenFY_SSC10, FourcartLai_ACHA09, GeJYe_MP11, LaiWang_SIOPT11}. 

When $p>1$, it is well known that the $p$-norm formulation will not lead to sparse solutions. However, to the best of our knowledge, a formal justification of this fact for a general setting with an arbitrary $p>1$ is not available, except an intuitive and straightforward geometric interpretation for special cases, e.g., basis pursuit. Besides, when different norms are used in objective functions of optimization problems, e.g., the ridge regression and elastic net, it is difficult to obtain a simple geometric interpretation.
Moreover,  for an arbitrary $p>1$, there lacks a {\em quantitative} characterization of how less sparse such solutions are and how these less sparse solutions depend on a measurement matrix  and a measurement vector, in comparison with the related problems for $0<p\le 1$.
In addition to theoretical interest,  these questions are also of practical values, since the $p$-norm based optimization with $p>1$ and its matrix norm extensions find applications in graph optimization \cite{ElChamie_ACC14}, machine learning, and image processing \cite{Lefkimiatis_TIP13}. It is also related to the $\ell_p$-programming with $p>1$ coined by  Terlaky \cite{Terlaky_EJOR85}.
Motivated by the aforementioned questions and their implications in applications,  
we give a formal argument for a broad class of  $p$-norm based optimization problems with $p>1$ generalized from sparse optimization and other fields. When $p>1$, our results  show that while these problems are smooth optimization problems, they not only fail to achieve sparse solutions but also yield the {\em least} sparse solutions generically. Specifically, when $p>1$, for almost all measurement matrices $A \in \mathbb R^{m\times N}$ and measurement vectors $y \in \mathbb R^{m}$, solutions to these $p$-norm based optimization problems have full support, i.e., the support size is $N$; see Theorems~\ref{thm:BP_p}, \ref{thm:RR_p}, \ref{thm:EN_p}, \ref{thm:BPDn_p}, and \ref{thm:BPDn02_p} for formal statements. The proofs for these results turn out to be nontrivial, since except $p=2$, the optimality conditions of these optimization problems yield highly nonlinear equations and there are no closed form expressions of optimal solutions in terms of $A$ and $y$. To overcome these technical difficulties, we exploit techniques from optimization and matrix analysis and give a systematic treatment to a broad class of $p$-norm based optimization problems originally from sparse optimization and other related fields,
%
%
including generalized  basis pursuit, basis pursuit denoising, ridge regression, and elastic net.
The results developed in this paper will also deepen the understanding of general $p$-norm based optimization problems emerging from many applications and shed light on their computation and  numerical  analysis.

The rest of the paper is organized as follows. In Section~\ref{sect:formulation}, we introduce generalized $p$-norm based optimization problems arising from applications and show the solution existence and uniqueness. When $p>1$, a lower sparsity bound and other preliminary results are established in Section~\ref{sect:preliminary}. Section~\ref{sect:least_sparsity} develops the main results of the paper, namely, the least sparsity of $p$-norm optimization based generalized basis pursuit,  generalized ridge regression and elastic net, and generalized basis pursuit denoising for $p>1$. In Section~\ref{sect:comparison},  we extend the least sparsity results for $p>1$ to measurement vectors restricted to a subspace of the range of $A$ and compare this result with the related $\ell_p$-optimization for $0<p\le 1$ arising from compressed sensing. Finally, conclusions are made in Section~\ref{sect:conclusion}.


{\it Notation}. Let $A=[a_1, \ldots, a_N]$ be an $m\times N$ real matrix with $N>m$, where $a_i\in \mathbb R^m$ denotes the $i$th column of $A$. For a given vector $x\in \mathbb R^n$, $\mbox{supp}(x)$ denotes the support of $x$.
For any index set $\Ical \subseteq \{1, \ldots, N\}$, let $|\Ical|$ denote the cardinality of $\Ical$, and $A_{\bullet\Ical}=[ a_i ]_{i\in \Ical}$ be the submatrix of $A$ formed by the columns of $A$ indexed by elements of $\Ical$. For a given matrix $M$, $R(M)$ denotes the range of $M$.
Let $\mbox{sgn}(\cdot)$ denote the signum function with $\mbox{sgn}(0):=0$.
Let $\succcurlyeq$ denote the positive semi-definite order, i.e., for two real symmetric matrices $P$ and $Q$, $P \succcurlyeq Q$ means that $(P-Q)$ is positive semi-definite.
The gradient of a real-valued differentiable function $f:\mathbb R^n \rightarrow \mathbb R$ is given by $\nabla f(x)=\big(\frac{\partial f(x)}{\partial x_1}, \ldots, \frac{\partial f(x)}{\partial x_n} \big)^T \in \mathbb R^n$. Let $F:\mathbb R^{n}\times \mathbb R^r \rightarrow \mathbb R^s$ be a differentiable function given by $F(x, z)=(F_1(x, z), \ldots, F_s(x, z))^T$ with $F_i:\mathbb R^n \times \mathbb R^r \rightarrow \mathbb R$ for $i=1,\ldots, s$. The Jacobian of $F$ with respect to $x=(x_1, \ldots, x_n)^T\in \mathbb R^n$ is
\[
  \mathbf J_x F(x, z)
  \, = \, \begin{bmatrix} \frac{\partial F_1(x, z)}{\partial x_1} & \cdots & \cdots \frac{\partial F_1(x, z)}{\partial x_n} \\  \frac{\partial F_2(x, z)}{\partial x_1} & \cdots & \cdots \frac{\partial F_2(x, z)}{\partial x_n} \\  \vdots & & \vdots \\ \frac{\partial F_s(x, z)}{\partial x_1} & \cdots & \cdots \frac{\partial F_s(x, z)}{\partial x_n} \end{bmatrix} \in \mathbb R^{s \times n}.
\]
By convention, we also use $\nabla_x F(x, z)$ to denote $\mathbf J_x F(x, z)$. Finally, by saying that a statement (P) holds for almost all $x$ in a finite dimensional real vector space $E$, we mean that (P) holds on a set $W \subseteq E$ whose complement $W^c$ has zero Lebesgue measure.

%
\section{Generalized $p$-norm based Optimization Problems} \label{sect:formulation}


In this section, we introduce a broad class of widely studied $p$-norm based optimization problems emerging from sparse optimization, statistics and other fields, and we discuss their generalizations. Throughout this section, we let the constant $p>0$, the matrix $A\in \mathbb R^{m\times N}$ and the vector $y \in \mathbb R^m$. For any $p>0$ and $x = (x_1, \ldots, x_N)^T \in \mathbb R^N$, define $\| x \|_p := \big(\sum^N_{i=1} | x_i|^p\big)^{1/p}$.

\gap

\noindent $\bullet$ {\bf Generalized Basis Pursuit}. Consider the following linear equality constrained optimization problem whose objective function is given by the $p$-norm (or quasi-norm):
\begin{equation} \label{eqn:BP_p_norm}
\text{BP}_p: \quad  \underset{x \in \mathbb R^N}{\text{min}} \ \|x\|_p  \quad \text{subject to} \quad Ax=y,
%
%
\end{equation}
where $y\in R(A)$. Geometrically, this problem seeks to minimize the $p$-norm distance  from the origin to the affine set defined by $Ax = y$. When $p=1$, it becomes the standard basis pursuit \cite{CandesRT_TIT06, ChenDS_SIAMREV01, FoucartRauhut_book2013}.

\noindent $\bullet$ {\bf Generalized Basis Pursuit Denoising}. Consider the following constrained optimization problem which incorporates noisy signals:
\begin{equation} \label{eqn:BP_denoising01}
 \text{BPDN}_p: \quad \underset{x \in \mathbb R^N}{\text{min}} \ \| x \|_p \ \quad \text{subject to} \quad \| A x - y \|_2 \le \varepsilon,
\end{equation}
where $\varepsilon > 0$ characterizes the bound of noise or errors.
%
%
When $p=1$, it becomes the standard basis pursuit denoising (or quadratically constrained basis pursuit)  \cite{Bryan_SIREW13, FoucartRauhut_book2013, VanDenBergF_SSC08}.
Another  version of the generalized basis pursuit denoising is given by the following optimization problem:
\begin{equation} \label{eqn:BP_denoising02}
  \underset{x \in \mathbb R^N}{\text{min}} \ \| A x - y \|_2 \ \quad \text{subject to} \quad \| x \|_p \le \eta,
\end{equation}
where the bound $\eta>0$. Similarly, when $p=1$, the optimization problem (\ref{eqn:BP_denoising02}) pertains to a relevant formulation of basis pursuit denoising \cite{FoucartRauhut_book2013,  VanDenBergF_SSC08}.

\noindent $\bullet$ {\bf Generalized Ridge Regression and Elastic Net}.  Consider the following unconstrained optimization problem:
%
%
\begin{equation} \label{eqn:ridge_reg}
  \text{RR}_p: \quad \underset{x \in \mathbb R^N}{\text{min}} \ \frac{1}{2} \| A x - y \|^2_2 + \lambda \, \| x \|^{p}_p,
\end{equation}
where $\lambda >0$ is the penalty parameter. When $p=2$, it becomes the standard ridge regression extensively studied in statistics \cite{HastieTF_book09, HoerlK_Techno70}; when $p=1$, it yields the least absolute shrinkage and selection operator (LASSO) with the $\ell_1$-norm penalty \cite{RTibshirani_EJS13, RTibshirani_JRSS96}. A related optimization problem is the generalized elastic net arising from statistics:
\begin{equation} \label{eqn:elastic_net}
 \text{EN}_p: \quad \underset{x \in \mathbb R^N}{\text{min}} \ \frac{1}{2} \| A x - y \|^2_2 + \lambda_1 \, \| x \|^r_p + \lambda_2 \, \| x \|^2_2,
\end{equation}
where $r>0$ and $\lambda_1, \lambda_2$ are positive penalty parameters. When $p=r=1$, the EN$_p$ (\ref{eqn:elastic_net}) becomes the standard elastic net formulation which combines the $\ell_1$ and $\ell_2$ penalties in regression \cite{Zhou_JRSS05}. Moreover, if we allow $\lambda_2$ to be non-negative, then the RR$_p$ (\ref{eqn:ridge_reg}) can be treated as a special case of the EN$_p$ (\ref{eqn:elastic_net}) with $r=p$, $\lambda=\lambda_1>0$, and $\lambda_2=0$.

%

In the sequel, we show the existence and uniqueness of optimal solutions for the generalized optimization problems introduced above.

\begin{proposition} \label{prop:solution_existence_p_norm}
 Fix an arbitrary $p>0$. Each of the optimization problems (\ref{eqn:BP_p_norm}), (\ref{eqn:BP_denoising01}), (\ref{eqn:BP_denoising02}), (\ref{eqn:ridge_reg}), and (\ref{eqn:elastic_net})  attains an optimal solution for any given $A$, $y$, $\varepsilon>0$, $\eta>0$, $\lambda>0$, $r>0$, $\lambda_1>0$ and $\lambda_2\ge 0$ as long as the associated constraint sets are nonempty. Further, when $p>1$, each of (\ref{eqn:BP_p_norm}), (\ref{eqn:BP_denoising01}),  and (\ref{eqn:ridge_reg}) has a unique optimal solution. Besides, when $p\ge 1$, $r \ge 1$, $\lambda_1>0$, and $\lambda_2>0$, (\ref{eqn:elastic_net}) has a unique optimal solution.
\end{proposition}

\begin{proof}
For any $p>0$, the optimization problems (\ref{eqn:BP_p_norm}), (\ref{eqn:BP_denoising01}), (\ref{eqn:ridge_reg}), and (\ref{eqn:elastic_net}) attain optimal solutions since their objective functions are continuous and coercive and the constraint sets (if exist) are closed. The optimization problem (\ref{eqn:BP_denoising02}) also attains a solution because it has a continuous objective function and a compact constraint set.

 When $p \ge 1$, (\ref{eqn:BP_p_norm}) and (\ref{eqn:BP_denoising01}) are convex optimization problems, and they are equivalent to  $\min_{A x = y } \, \| x \|^p_p$ and $\min_{\| A x - y\|_2 \le \varepsilon } \, \| x \|^p_p$ respectively.  Note that when $p>1$, the function $\| \cdot \|^p_p$ is strictly convex on $\mathbb R^N$; see the proof in Appendix (c.f. Section~\ref{sect:appendix}). Therefore, each of (\ref{eqn:BP_p_norm}),  (\ref{eqn:BP_denoising01}) and (\ref{eqn:ridge_reg}) has a unique optimal solution.
When $p\ge 1$, $r \ge 1$, $\lambda_1>0$, and $\lambda_2>0$, the generalized elastic net (\ref{eqn:elastic_net}) is a convex optimization problem with a strictly convex objective function and thus has a unique optimal solution.
\end{proof}

%
\section{Preliminary Results on Sparsity of $p$-norm based Optimization with $p>1$} \label{sect:preliminary}

%
%

This section develops key preliminary results for the global sparsity analysis of $p$-norm based optimization problems when $p>1$.

%
\subsection{Lower Bound on Sparsity of $p$-norm based Optimization with $p>1$} 

We first establish a lower bound on the sparsity of optimal solutions arising from the $p$-norm based optimization with $p>1$. Specifically, we show that when $p>1$, for almost all $(A, y) \in \mathbb R^{m\times N} \times \mathbb R^m$, any (nonzero) optimal solution has at least $(N-m+1)$ nonzero elements and thus is far from sparse when $N \gg m$.  This result is critical to show in the subsequent section that for almost all $(A, y)$, an optimal solution achieves a full support; see the proofs of Propositions~\ref{prop:BP_p>2} and \ref{prop:full_support_BPDN_P>2}, and Theorem~\ref{thm:RR_p}.
Toward this end, we define the following set in $\mathbb R^{m\times N} \times \mathbb R^m$ with $N\ge m$:
\begin{equation} \label{eqn:set_S}
   S \, := \, \Big\{ (A, y) \in \mathbb R^{m\times N} \times \mathbb R^m \, | \, \mbox{ every $m\times m$ submatrix of $A$ is invertible, \ and } \ y \ne 0 \Big \}.
\end{equation}
Clearly, $S$ is open and its complement $S^c$ has zero measure in $\mathbb R^{m\times N} \times \mathbb R^m$. Note that a matrix $A$ satisfying the condition in (\ref{eqn:set_S}) is called to be of completely full rank \cite{LaiWang_SIOPT11}.
 To emphasize the dependence of optimal solutions on the measurement matrix $A$ and the measurement vector $y$, we write an optimal solution as $x^*_{(A, y)}$ or $x^*(A, y)$ below; the latter notation is used when $x^*$ is unique for any given $(A, y)$ so that $x^*$ is a function of $(A, y)$.

\begin{proposition} \label{prop:sparsity_lower_bound}
  Let $p>1$. For any $(A, y)\in S$, the following statements hold:
  \begin{itemize}
    \item [(i)] The optimal solution $x^*_{(A, y)}$ to the BP$_p$ (\ref{eqn:BP_p_norm}) satisfies $| \mbox{supp}(x^*_{(A, y)})| \ge N-m+1$;
    \item [(ii)] If $0<\varepsilon<\| y \|_2$, then the optimal solution $x^*_{(A, y)}$ to the BPDN$_p$ (\ref{eqn:BP_denoising01}) satisfies $| \mbox{supp}(x^*_{(A, y)})| \ge N-m+1$;
    \item [(iii)] For any $\lambda>0$, the optimal solution $x^*_{(A, y)}$ to the RR$_p$ (\ref{eqn:ridge_reg})   satisfies $| \mbox{supp}(x^*_{(A, y)})| \ge N-m+1$;
    \item [(iv)] For any $r>0$, $\lambda_1>0$ and $\lambda_2 \ge 0$, each nonzero optimal solution $x^*_{(A, y)}$ to the EN$_p$ (\ref{eqn:elastic_net}) satisfies $| \mbox{supp}(x^*_{(A, y)})| \ge N-m+1$.
  \end{itemize}
\end{proposition}

We give two remarks on the conditions stated in the proposition before presenting its proof:

 (a) Note that if $\varepsilon \ge \| y\|_2$ in statement (ii), then $x=0$ is feasible such that the BPDN$_p$ (\ref{eqn:BP_denoising01}) attains the trivial (unique) optimal solution $x^*=0$. For this reason, we impose the assumption $0<\varepsilon<\| y \|_2$.

 (b) When $0<r<1$ in statement (iv) with $\lambda_1>0$ and $\lambda_2 \ge 0$, the EN$_p$ (\ref{eqn:elastic_net}) has a non-convex objective function and it may have multiple optimal solutions. Statement (iv) says that any of such nonzero optimal solutions has the sparsity of at least $N-m+1$.

\begin{proof}
  Fix $(A, y) \in S$. We write an optimal solution $x^*_{(A, y)}$ as $x^*$ for notational simplicity in the proof. Furthermore, let $f(x):= \| x \|^p_p$. Clearly, when $p>1$, $f$ is continuously differentiable on $\mathbb R^N$.

 (i) Consider the BP$_p$ (\ref{eqn:BP_p_norm}).
  Note that $0 \ne y \in R(A)$ for any $(A, y) \in S$. By Proposition~\ref{prop:solution_existence_p_norm}, the BP$_p$ (\ref{eqn:BP_p_norm}) (\ref{eqn:BP_p_norm}) has a unique optimal solution $x^*$ for each $(A, y) \in S$. In view of $ x^* = \mbox{argmin}_{Ax = y} f(x)$, the necessary and sufficient optimality condition for $x^*$ is given by the following KKT condition:
 \[
    \nabla f(x^*) - A^T \nu =0, \qquad  A \, x^* = y,
 \]
 where $\nu \in \mathbb R^m$ is the Lagrange multiplier, and $(\nabla f(x))_i= p \cdot \mbox{sgn}(x_i) \cdot | x_i |^{p-1}$ for each $i=1, \ldots, N$. Note that $\nabla f(x)$ is positively homogeneous in $x$ and each $(\nabla f(x))_i$ depends on $x_i$ only.
 Suppose that $x^*$ has at least $m$ zero elements. Hence,  $\nabla f(x^*)$ has at least $m$ zero elements. By the first equation in the KKT condition, we deduce that there is an $m\times m$ submatrix $A_1$ of $A$ such that $A^T_1 \nu=0$. Since $A_1$ is invertible, we have $\nu=0$ such that $\nabla f(x^*)=0$. This further implies that $x^*=0$. This contradicts $A \, x^* = y \ne 0$. Therefore, $| \mbox{supp}(x^*)| \ge N-m+1$  for all $(A, y) \in S$.

 (ii) Consider the BPDN$_p$ (\ref{eqn:BP_denoising01}). Note that for any given $(A, y)\in S$ and $0<\varepsilon < \| y \|_2$, the  BPDN$_p$  (\ref{eqn:BP_denoising01}) has a unique nonzero optimal solution $x^*$.  Let $g(x):=\| A x - y\|^2_2 - \varepsilon^2$. Since $A$ has full row rank, there exists $\ol x \in \mathbb R^N$ such that $g(\ol x)<0$. As  $g(\cdot)$ is a convex function, the Slater's constraint qualification holds for the equivalent convex optimization problem $\min_{g(x) \le 0} \, f(x)$. Hence $x^*$  satisfies the KKT condition with the Lagrange multiplier $\mu\in \mathbb R$, where $\perp$ denotes the orthgonality,
 \[
    \nabla f(x^*) + \mu \, \nabla g(x^*) = 0, \qquad  0 \le \mu \perp g(x^*) \le 0.
 \]
 We claim that $\mu>0$. Suppose not. Then it follows from the first equation in the KKT condition that $\nabla f(x^*)=0$, which implies $x^*=0$. This yields $g(x^*)= \| y\|^2_2 - \varepsilon^2>0$, contradiction. Therefore $\mu>0$ such that $g(x^*)=0$. Using $\nabla g(x^*)= 2 A^T (A x^* - y)$, we have $\nabla f(x^*) + 2 \mu A^T (A x^* - y)  = 0$. Suppose, by contradiction, that $x^*$ has at least $m$ zero elements. Without loss of generality, we assume that the first $m$ elements of $x^*$ are zeros. Partition the matrix $A$ into $A=[A_1 \ A_2]$, where $A_1 \in \mathbb R^{m\times m}$ and $A_2 \in \mathbb R^{m\times (N-m)}$. Similarly, $x^*=[0; \, \wt x^*]$, where $\wt x^* \in \mathbb R^{N-m}$. Hence, the first $m$ elements of $\nabla f(x^*)$ are zero. By the first equation in the KKT condition, we derive $2\mu A^T_1 (A x^* - y)=0$. Since $\mu>0$ and $A_1$ is invertible, we obtain $A x^*- y =0$.
 This shows that $g(x^*)= -\varepsilon^2<0$, contradiction to $g(x^*)=0$.

 (iii) Consider the RR$_p$ (\ref{eqn:ridge_reg}). The unique optimal solution $x^*$ is characterized by the optimality condition: $A^T(Ax^*-y) + \lambda \, \nabla f(x^*) =0$, where $\lambda>0$.
 Suppose, by contradiction, that $x^*$ has at least $m$ zero elements.
 Using the similar argument for Case (ii), we derive that $A x^* - y =0$. In view of the optimality condition, we thus have $\nabla f(x^*)=0$. This implies that $x^*=0$. Substituting $x^*=0$ into the optimality condition yields $A^T y =0$. Since $A$ has full row rank, we obtain $y=0$. This leads to a contradiction. Hence $| \mbox{supp}(x^*)| \ge N-m+1$ for all $(A, y) \in S$.

(iv) Consider the EN$_p$ (\ref{eqn:elastic_net}) with the exponent $r>0$ and the penalty parameters $\lambda_1>0$ and $\lambda_2 \ge 0$.
When $\lambda_2 = 0$, it is closely related to the RR$_p$ with the exponent on $\| x\|_p$ replaced by an arbitrary $r>0$. For any $(A, y) \in S$, let $x^*$ be a (possibly non-unique) nonzero optimal solution which satisfies the optimality condition: $A^T(Ax^*-y) + r \lambda_1 \cdot \| x^* \|^{r-1}_p \cdot \nabla \| x^* \|_p + 2\lambda_2 x^* =0$, where for any nonzero $x \in \mathbb R^N$,
 \[
   \nabla \| x \|_p = \frac{1}{\| x\|^{p-1}_p} \, \Big( \, \mbox{sgn}(x_1) |x_1|^{p-1}, \ \ldots, \ \mbox{sgn}(x_N) |x_N|^{p-1} \, \Big)^T = \frac{\nabla \| x \|^p_p }{p \cdot \| x\|^{p-1}_p} .
 \]
The optimality condition can be equivalently written as
\begin{equation} \label{eqn:EN_optimality}
  p \cdot A^T(Ax^*-y) + r \lambda_1 \cdot \| x^* \|^{r-p}_p \cdot \nabla f(x^*) + 2p \lambda_2 \, x^* =0.
\end{equation}
 Consider two cases: (iv.1) $\lambda_2=0$. By the similar argument for Case (iii), it is easy to show that $| \mbox{supp}(x^*)| \ge N-m+1$ for all $(A, y) \in S$; (iv.2) $\lambda_2>0$. In this case,  suppose, by contradiction, that $x^*$ has at least $m$ zero elements. As before, let $A=[A_1 \ A_2]$ and $x^*=[0; \wt x^*]$ with $A_1 \in \mathbb R^{m\times m}$ and $\wt x^* \in \mathbb R^{N-m}$. Hence, the optimality condition leads to $p \cdot A^T_1 (A x^*-y)=0$. This implies that $A x^*-y=0$ such that $r \lambda_1 \cdot \| x^* \|^{r-p}_p \cdot \nabla f(x^*) + 2p \lambda_2 \, x^* =0$.  Since $(\nabla f(x))_i= p \cdot \mbox{sgn}(x_i) \cdot | x_i |^{p-1}$ for each $i=1, \ldots, N$, we obtain $r \lambda_1 \cdot \| x^* \|^{r-p}_p \cdot | x^*_i|^{p-1} + 2 \lambda_2 |x^*_i|=0$ for each $i$. Hence $x^*=0$,  a contradiction.
  We thus conclude that $| \mbox{supp}(x^*)| \ge N-m+1$ for all $(A, y) \in S$.
\end{proof}

We discuss an extension of the sparsity lower bound developed in Proposition~\ref{prop:sparsity_lower_bound} to another formulation of the basis pursuit denoising given in (\ref{eqn:BP_denoising02}).
 It is noted that if $\eta \ge \min_{A x = y} \|x\|_p$ (which implies $y \in R(A)$), then the optimal value of (\ref{eqn:BP_denoising02}) is zero and can be achieved at some feasible $x^*$ satisfying $A x^* = y$.  Hence any optimal solution $x'$ must satisfy $A x' = y$ so that the optimal solution set is given by $\{ x \in \mathbb R^N \, | \, A x = y, \ \| x \|_p \le \eta \}$, which is closely related to the BP$_p$ (\ref{eqn:BP_p_norm}). This means that if $\eta \ge \min_{A x = y} \|x\|_p$, the optimization problem (\ref{eqn:BP_denoising02}) can be converted to a reduced and simpler problem. For this reason, we assume that $0<\eta < \min_{A x = y} \|x\|_p$ for (\ref{eqn:BP_denoising02}). The following proposition presents important results under this assumption; these results will be used for the proof of Theorem~\ref{thm:BPDn02_p}.

\begin{proposition} \label{prop:BPDN2_lower_bound}
 The following hold for the optimization problem (\ref{eqn:BP_denoising02}) with $p>1$:
\begin{itemize}
  \item [(i)] If $A$ has full row rank and $0<\eta < \min_{A x = y} \|x\|_p$, then  (\ref{eqn:BP_denoising02}) attains a unique optimal solution with a unique positive Lagrange multiplier;
  \item [(ii)] For any $(A, y)$ in the set $S$ defined in (\ref{eqn:set_S}) and $0<\eta < \min_{A x = y} \|x\|_p$, the unique optimal solution $x^*_{(A, y)}$ satisfies $| \mbox{supp}(x^*_{(A, y)})| \ge N-m+1$.
\end{itemize}
\end{proposition}

%
%

\begin{proof}
  (i) Let $A$ be of full row rank. Hence, $y \in R(A)$ so that $\eta$ is well defined.
  Let $x^*$ be an arbitrary optimal solution to (\ref{eqn:BP_denoising02}) with the specified $\eta>0$. Hence $x^*=\mbox{argmin}_{f(x) \le \eta^p} \, \frac{1}{2} \| A x - y \|^2_2$, where we recall that $f(x)=\| x \|^p_p$. Clearly, the Slater's constraint qualification holds for the convex optimization problem (\ref{eqn:BP_denoising02}). Therefore, $x^*$ satisfies the following KKT condition:
\begin{equation} \label{eqn:BPDN02_KKT}
   A^T (A x^* - y) + \mu \nabla f(x^*)=0, \qquad  0\le \mu \perp f(x^*)- \eta^p \le 0,
\end{equation}
where $\mu \in \mathbb R$ is the Lagrange multiplier. We claim that $\mu$ must be positive. Suppose not, i.e., $\mu=0$. By the first equation in (\ref{eqn:BPDN02_KKT}), we obtain $A^T (A x^* - y)=0$. Since $A$ has full row rank, we have $A x^* = y$. Based on the assumption on $\eta$, we further have $\| x^*\|_p>\eta$, contradiction to $f(x^*)\le \eta^p$. This proves the claim. Since $\mu>0$,  it follows from the second equation in (\ref{eqn:BPDN02_KKT}) that any optimal solution $x^*$ satisfies $f(x^*)=\eta^p$ or equivalently $\| x ^* \|_p = \eta$.
To prove the uniqueness of optimal solution, suppose, by contradiction, that $x^*$ and $x'$ are two distinct optimal solutions for the given $(A, y)$. Thus $\| x^*\|_p=\| x'\|_p=\eta$. Since (\ref{eqn:BP_denoising02}) is a convex optimization problem, the optimal solution set is convex so that $\lambda x^*+(1-\lambda) x'$ is an optimal solution for any $\lambda \in[0, 1]$. Hence, $\|\lambda x^*+(1-\lambda) x'\|_p=\eta, \forall \, \lambda \in[0, 1]$. Since  $\|\cdot \|^p_p$ is strictly convex  when $p>1$, we have $\eta^p = \|\lambda x^*+(1-\lambda) x'\|^p_p < \lambda \|x^*\|^p_p +(1-\lambda)\|x'\|^p_p = \eta^p$ for each $\lambda \in (0, 1)$. This yields a contradiction. We thus conclude that (\ref{eqn:BP_denoising02}) attains a unique optimal solution with $\mu>0$.

(ii) Let $(A, y) \in S$. Clearly, $A$ has full row rank so that (\ref{eqn:BP_denoising02}) has a unique optimal solution $x^*$ with a positive Lagrange multiplier $\mu$. Suppose $x^*$ has at least $m$ zero elements. It follows from the first equation in (\ref{eqn:BPDN02_KKT}) and the similar argument for Case (iii) of Proposition~\ref{prop:sparsity_lower_bound} that $A x^* = y$. In light of the assumption on $\eta$, we have $\| x^* \|_p> \eta$, a contradiction. Therefore $| \mbox{supp}(x^*)| \ge N-m+1$ for any $(A, y) \in S$.
\end{proof}

%
\subsection{Technical Result on Measure of the Zero Set of $C^1$-functions} \label{subsect:measure} 

%
%

As shown in Proposition~\ref{prop:solution_existence_p_norm},  when $p>1$, each of the BP$_p$ (\ref{eqn:BP_p_norm}),  BPDN$_p$ (\ref{eqn:BP_denoising01}), and RR$_p$  (\ref{eqn:ridge_reg}) has a unique optimal solution $x^*$ for any given $(A, y)$. Under additional conditions, each of the EN$_p$ (\ref{eqn:elastic_net}) and the optimization problem (\ref{eqn:BP_denoising02}) also attains a unique optimal solution.
Hence, for each of these problems, the optimal solution $x^*$ is a function of $(A, y)$, and each component of $x^*$ becomes a real-valued function $x^*_i(A, y)$. Therefore, the global sparsity of $x^*$ can be characterized by the zero set of each $x^*_i(A, y)$. The following technical lemma gives a key result on the measure of the zero set of a real-valued $C^1$-function under a suitable assumption.

\begin{lemma} \label{lem:measure_of_zero_set}
 Let $f:\mathbb R^n \rightarrow \mathbb R$ be continuously differentiable (i.e., $C^1$)  on an open set $W \subseteq \mathbb R^n$ whose complement $W^c$ has zero measure in $\mathbb R^n$. Suppose $\nabla f(x) \ne 0$ for any $x \in W$ with $f(x)=0$. Then the zero set
 $f^{-1}(\{0\}):=\{x \in \mathbb R^n \, | \, f(x)=0 \}$ has zero measure.
\end{lemma}

\begin{proof}
 Consider an arbitrary $x^*\in W$. If $f(x^*)=0$, then $\nabla f(x^*) \ne 0$.
 Without loss of generality, we assume that $\frac{\partial f}{\partial x_n}(x^*) \ne 0$. Let $z:=(x_1, \ldots, x_{n-1})^T \in \mathbb R^{n-1}$.  By the implicit function theorem, there exist an open neighborhood $U \subset \mathbb R^{n-1}$ of $z^*:=(x^*_1, \ldots, x^*_{n-1})^T$, an open neighborhood $V\subset \mathbb R$ of $x^*_n$, and a unique $C^1$ function $g:U\rightarrow V$ such that $f(z, g(z))=0$ for all $z \in U$. The set $f^{-1}(\{0\})\cap (U\times V) = \{ (z, g(z) ) \, | \, z \in U \}$ has zero measure in $\mathbb R^{n}$ since it is an $(n-1)$ dimensional manifold in the open set $U\times V \subset \mathbb R^n$. Moreover, in view of the continuity of $f$, we deduce that for any $x^* \in W$ with $f(x^*) \ne 0$, there exists an open set $\mathcal B(x^*)$ of $x^*$ such that $f(x) \ne 0, \forall \, x \in \mathcal B(x^*)$.
Combining these results, it is seen that for any $x \in W$, there exists an open set $\mathcal B(x)$  of $x$ such that $f^{-1}(\{0\})\cap \mathcal B(x)$ has zero measure.
Clearly, the union of these open sets given by $\bigcup_{x \in W} \mathcal B(x)$ forms an open cover of $W$. Since $\mathbb R^n$ is a topologically separable metric space, so is $W \subset \mathbb R^n$ and thus it is a Lindel\"{o}f space \cite{Munkres_book00, Royden_book88}. Hence such an open cover attains a countable sub-cover $\bigcup_{i\in \mathbb N} \mathcal B(x^i)$ of $W$, where each $x^i \in W$.
Since $f^{-1}(\{0\}) \cap \mathcal B(x^i)$ has zero measure for each $i \in \mathbb N$,
 the set $W\cap f^{-1}(\{0\})$ has zero measure. Besides, since $f^{-1}(\{0\}) \subseteq W^c \cup \big( W \cap f^{-1}(\{0\}) \big)$ and both $W^c$ and $W\cap f^{-1}(\{0\})$ have zero measure, we conclude that $f^{-1}(\{0\})$ has zero measure.
\end{proof}


%
\section{Least Sparsity of $p$-norm based Optimization Problems with $p>1$} \label{sect:least_sparsity}

In this section, we establish the main results of the paper, namely, when $p>1$, the $p$-norm based optimization problems yield {\em least} sparse solutions for almost all $(A, y)$. We introduce more notation to be used through this section. Let $f(x):=\| x \|^p_p$ for $x \in \mathbb R^N$, and when $p>1$, we define for each $z \in \mathbb R$,
\begin{equation} \label{eqn:g_h_def}
   g(z) \, := \, p \cdot \mbox{sgn}(z) \cdot |z|^{p-1}, \qquad h(z) \, := \, \mbox{sgn}\left(z \right) \cdot \left |\frac{z}{p} \right|^{\frac{1}{p-1}},
\end{equation}
where $\mbox{sgn}(\cdot)$ denotes the signum function with $\mbox{sgn}(0):=0$. Direct calculation shows that (i) when $p>1$, $g(z)=(|z|^p)', \forall \, z \in \mathbb R$, and $h(z)$ is the inverse function of $g(z)$; (ii) when $p \ge 2$, $g$ is continuously differentiable and $g'(z)= p (p-1) \cdot |z|^{p-2}, \forall \, z \in \mathbb R$; and (iii) when $1<p \le 2$, $h$ is continuously differentiable and $h'(z)=  |z|^{\frac{2-p}{p-1}}/[(p-1) \cdot p^{1/(p-1)}], \forall \, z \in \mathbb R$. Furthermore, when $p>1$,  $\nabla f(x)=\big( g(x_1), \ldots, g(x_N) \big)^T$.

The proofs for the least sparsity developed in the rest of the section share the similar methodologies. To facilitate the reading, we give an overview of main ideas of these proofs and comment on certain key steps in the proofs. As indicated at the beginning of Section~\ref{subsect:measure}, the goal is to show that the zero set of each component of an optimal solution $x^*$, which is a real-valued function of $(A, y)$, has zero measure. To achieve this goal, we first show using the KKT conditions and the implicit function theorem that $x^*$, possibly along with a Lagrange multiplier if applicable, is a $C^1$ function of $(A, y)$ on a suitable open set $S^{\,\prime}$ in $\mathbb R^{m\times N} \times \mathbb R^m$ whose complement has zero measure. We then show that for each $i=1, \ldots, N$, if $x^*_i$ is vanishing at $(A, y) \in S^{\,\prime}$, then its gradient evaluated at $(A, y)$ is nonzero. In view of Lemma~\ref{lem:measure_of_zero_set}, this leads to the desired result. Moreover, for each of the generalized optimization problems with $p>1$, i.e., the BP$_p$, BPDN$_p$, RR$_p$, and EN$_p$, we divide their proofs into two separate cases: (i) $p\ge 2$; and (ii) $1<p \le 2$. This is because each case invokes the derivative of $g(\cdot)$ or its inverse function $h(\cdot)$ defined in (\ref{eqn:g_h_def}). When $p \ge 2$, the derivative $g'(\cdot)$ is globally well defined. On the contrary, when $1<p\le 2$, $g'(\cdot)$ is not defined at zero. Hence, we use $h(\cdot)$ instead, since $h'(\cdot)$ is globally well defined in this case. The different choice of $g$ or $h$ gives rise to different arguments in the following proofs, and the proofs for $1<p \le 2$ are typically more involved.

%
\subsection{Least Sparsity of the Generalized Basis Pursuit with $p>1$} 

We consider the cases where $p \ge 2$ first.

\begin{proposition} \label{prop:BP_p>2}
 Let $p \ge 2$ and $N \ge m$. For almost all $(A, y) \in \mathbb R^{m\times N} \times \mathbb R^m$, the unique optimal solution $x^*_{(A, y)}$ to the BP$_p$ (\ref{eqn:BP_p_norm}) satisfies $| \mbox{supp}(x^*_{(A, y)})| = N$.
\end{proposition}

\begin{proof}
Recall that for any $(A, y) \in \mathbb R^{m\times N} \times \mathbb R^m$, the necessary and sufficient optimality condition for $x^*$ is given by the following KKT condition given in Proposition~\ref{prop:sparsity_lower_bound}:
 \[
    \nabla f(x^*) - A^T \nu =0, \qquad  A \, x^* = y,
 \]
 where $\nu \in \mathbb R^m$ is the Lagrange multiplier, and $(\nabla f(x))_i= p \cdot \mbox{sgn}(x_i) \cdot | x_i |^{p-1} = g(x_i)$ for each $i=1, \ldots, N$, where the function $g$ is defined in (\ref{eqn:g_h_def}).
 We show below that for any $(A^\diam, y^\diam)$ in the open set $S$ defined in (\ref{eqn:set_S}), $x^*(A, y)$ is continuously differentiable at $(A^\diam, y^\diam)$ and that each $x^*_i$ with $x^*_i(A^\diam, y^\diam)=0$ has  nonzero gradient at $(A^\diam, y^\diam)$.

%

 Recall that $x^*$ is unique for any $(A, y)$. Besides, for each $(A, y)\in S$, $A^T$ has full column rank such that the Lagrange multiplier $\nu$ is also unique in view of the first equation of the KKT condition. Therefore, $(x^*, \nu)$ is a function of $(A, y) \in S$. For notational simplicity, let $x^\diam:=x^*(A^\diam, y^\diam)$ and $\nu^\diam:=\nu(A^\diam, y^\diam)$. Define the index set $\Jcal := \{ i \, | \, x^\diam_i \ne 0 \}$. By Proposition~\ref{prop:sparsity_lower_bound}, we see that $\mathcal J$ is nonempty and $|\Jcal^c| \le m-1$. Further, in light of the KKT condition, $(x^*, \nu)\in \mathbb R^N\times \mathbb R^m$ satisfies the following equation:
 \[
    F(x, \nu, A, y) \, := \, \begin{bmatrix} \nabla f(x) - A^T \nu \\ A x - y \end{bmatrix} \, = \, 0.
 \]
 Clearly, $F:\mathbb R^N \times \mathbb R^m \times \mathbb R^{m\times N} \times \mathbb R^m \rightarrow \mathbb R^{N+m}$ is $C^1$, and its Jacobian with respect to $(x, \nu)$ is
 \[
    \mathbf J_{(x, \nu)} F(x, \nu, A, y) \, = \, \begin{bmatrix} \Lambda(x) & -A^T \\ A & 0 \end{bmatrix},
 \]
 where the diagonal matrix $\Lambda(x) :=\mbox{diag}( g'(x_1), \ldots, g'(x_N))$. Partition $\Lambda^\diam:=\Lambda(x^\diam)$ and $A$ as $\Lambda^\diam= \mbox{diag}( \Lambda^\diam_1, \,  \Lambda^\diam_2)$ and $A = \big [ A_1 \ A_2 \big]$ respectively, where $\Lambda^\diam_{1} := \mbox{diag}( g'(x^\diam_i) )_{i\in \Jcal^c}=0$, $\Lambda^\diam_{2} := \mbox{diag}( g'(x^\diam_i) )_{i\in \Jcal}$ is positive definite, $A_1:=A_{\bullet \Jcal^c}$, and $A_2:=A_{\bullet \Jcal}$. We claim that the following matrix is invertible:
 \[
     W \, := \, \mathbf J_{(x, \nu)} F(x^\diam, \nu^\diam, A^\diam, y^\diam) = \begin{bmatrix} \Lambda^\diam_1 & 0 & -A^T_1 \\ 0 & \Lambda^\diam_2 & - A^T_2 \\ A_1 & A_2 & 0 \end{bmatrix} \in \mathbb R^{(N+m)\times (N+m)}.
 \]
 In fact, let $z:=[u_1; u_2; v]\in \mathbb R^{N+m}$ be such that $W z =0$. Since $\Lambda^\diam_{1}=0$ and $\Lambda^\diam_{2}$ is positive definite, we have $A^T_1 v =0$, $u_2 = (\Lambda^\diam_{2})^{-1} A^T_2 v$, and $A_1 u_1 + A_2 u_2 =0$. Therefore, $0= v^T (A_1 u_1 + A_2 u_2) = v^T A_2 (\Lambda^\diam_{2})^{-1} A^T_2 v$, which implies that $A^T_2 v=0$ such that $u_2=0$ and $A_1 u_1=0$. Since $|\Jcal^c| \le m-1$ and any $m\times m$ submatrix of $A$ is invertible, the columns of  $A_1$ are linearly independent such that $u_1=0$. This implies that $A^T v =0$. Since $A$ has full row rank, we have $v=0$ and thus $z=0$. This proves that $W$ is invertible.
 By the implicit function theorem, there are local $C^1$ functions $G_1, G_2, H$ such that $x^*=(x^*_{\Jcal^c}, x^*_\Jcal)=(G_1(A, y), G_2(A, y)):=G(A, y)$, $\nu=H(A, y)$, and $F(G(A, y), H(A, y), A, y)=0$ for all $(A, y)$ in a neighborhood of $(A^\diam, y^\diam)$.
%
%

By the chain rule, we have
\[
   \underbrace{\mathbf J_{(x, \nu)} F(x^\diam, \nu^\diam, A^\diam, y^\diam)}_{:=W} \cdot \begin{bmatrix} \nabla_y G_1(A^\diam, y^\diam) \\ \nabla_y G_2(A^\diam, y^\diam) \\ \nabla_y H(A^\diam, y^\diam) \end{bmatrix} + \mathbf J_{y} F(x^\diam, \nu^\diam, A^\diam, y^\diam) \, = \, 0,
\]
where
\[
  \mathbf J_{y} F(x^\diam, \nu^\diam, A^\diam, y^\diam) = \begin{bmatrix} 0 \\ 0 \\ -I \end{bmatrix},
  \qquad W^{-1} := P =\begin{bmatrix} P_{11} & P_{12} & P_{13} \\ P_{21} & P_{22} & P_{23} \\ P_{31} & P_{32} & P_{33} \end{bmatrix}.
\]
It is easy to verify that $\nabla_y G_1(A^\diam, y^\diam)= P_{13}$ and $P_{13} A_1 = I$ by virtue of $P W = I$. The latter equation shows that each row of $P_{13}$ is nonzero, so is each row of $\nabla_y G_1(A^\diam, y^\diam)$. Thus each row of $\nabla_{(A, y)} G_1(A^\diam, y^\diam)$ is nonzero.
%
%
 Hence, for each $i=1,\ldots, N$, $x^*_i(A, y)$ is $C^1$ on the open set $S$, and when  $x^*_i(A^\diam, y^\diam)=0$ at $(A^\diam, y^\diam)\in S$, its gradient is nonzero.
By Lemma~\ref{lem:measure_of_zero_set}, we see that for each $i=1, \ldots, N$, the zero set of $x^*_i(A, y)$ has zero measure. This shows that $| \mbox{supp}(x^*(A, y))| = N$ for almost all $(A, y) \in \mathbb R^{m\times N} \times \mathbb R^m$.
\end{proof}

%
%

The next result addresses the case where $1<p \le 2$. In this case, it can be shown that if $x^*_i$ is vanishing at some $(A^\diam, y^\diam)$ in a certain open set, then the gradient of $x^*_i$ evaluated at $(A^\diam, y^\diam)$ also vanishes. This prevents us from applying Lemma~\ref{lem:measure_of_zero_set} directly. To overcome this difficulty, we introduce a suitable function which has  exactly the same sign of $x^*_i$ and to which Lemma~\ref{lem:measure_of_zero_set} is applicable. This technique is also used in other proofs for $1<p\le 2$; see Theorems~\ref{thm:RR_p}, \ref{thm:EN_p}, \ref{thm:BPDn02_p}, and Proposition~\ref{prop:BPDN_p<2}.

%
%

\begin{proposition} \label{prop:BP_p<2}
  Let $1<p \le 2$ and $N \ge 2m-1$. For almost all $(A, y) \in \mathbb R^{m\times N} \times \mathbb R^m$, the unique optimal solution $x^*_{(A, y)}$ to the BP$_p$ (\ref{eqn:BP_p_norm}) satisfies $|\mbox{supp}(x^*_{(A, y)})| = N$.
\end{proposition}

\begin{proof}
  Let $\wt S$ be the set of all $(A, y) \in  \mathbb R^{m\times N}\times \mathbb R^m$ satisfying the following conditions: (i) $y \ne 0$; (ii) each column of $A$ is nonzero; and (iii) for any index set $\Ical \subseteq \{ 1, \ldots, N\}$ with $|\Ical^c|\ge m$ and $\mbox{rank}(A_{\bullet \Ical})<m$, $\mbox{rank}(A_{\bullet \Ical^c})=m$. Note that such $A$ must have full row rank, i.e., $\mbox{rank}(A)=m$. It is easy to see that $\wt S$ is open and its complement $\wt S^c$ has zero measure. Note that the set $S$ given in (\ref{eqn:set_S}) is a proper subset of $\wt S$.

 Let $A=[a_1, \ldots, a_N]$, where $a_i \in \mathbb R^m$ is the $i$th column of $A$. It follows from the KKT condition $\nabla f(x^*) - A^T \nu =0$ that $x^*_i=h(a^T_i \nu)$ for each $i=1, \ldots, N$, where the function $h$ is defined in (\ref{eqn:g_h_def}). Along with the equation $A x^*=y$, we obtain the following equation for $(\nu, A, y)$:
 \[
    F(\nu, A, y ) \, := \, \sum^N_{i=1} a_i h(a^T_i \nu)-y \, = \, 0,
 \]
  where $F:\mathbb R^m \times \mathbb R^{m\times N} \times \mathbb R^m \rightarrow \mathbb R^m$ is $C^1$ and its Jacobian with respect to $\nu$ is
  \[
      \mathbf J_\nu F(\nu, A, y) \, = \, \begin{bmatrix} a_1  & \cdots & a_N \end{bmatrix} \begin{bmatrix} h'(a^T_1 \nu) &  & &  \\ & h'(a^T_2 \nu) & &  \\ & & \ddots &  \\ & & &  h'(a^T_N \nu) \end{bmatrix} \begin{bmatrix} a^T_1 \\ \vdots \\ a^T_N \end{bmatrix} \in \mathbb R^{m \times m}.
  \]
  We show next that for any $(A^\diam, y^\diam)\in \wt S$ with (unique) $\nu$ satisfying $F(\nu, A^\diam, y^\diam)=0$, the Jacobian $Q:=\mathbf J_\nu F(\nu, A^\diam, y^\diam)$ is positive definite.
  To prove this result, we first note that $\nu \ne 0$ since otherwise $x^*=0$ so that $Ax^* =0=y$, contradiction to $y \ne 0$. Using the formula for $h'(\cdot)$ given below (\ref{eqn:g_h_def}),  we have
 \begin{equation} \label{eqn:wT*Q*w}
  w^T Q w \, = \, \sum^{N}_{i=1} (a^T_i w)^2 \cdot h'(a^T_i \nu) \, = \, \frac{1}{(p-1)\cdot p^{1/(p-1)} } \sum^{N}_{i=1} \big(a^T_i w \big)^2 \cdot \big|a^T_i \nu \big|^{\frac{2-p}{p-1}}, \qquad \forall \, w \in \mathbb R^m.
 \end{equation}
 Clearly, $Q$ is positive semi-definite. Suppose, by contradiction, that there exists $w\ne 0$ such that $w^T Q w=0$. Define the index set $\Ical:=\{ i \, | \, a^T_i w =0 \}$. Note that $\mathcal I$ must be nonempty because otherwise, it follows from (\ref{eqn:wT*Q*w}) that $A^T \nu=0$, which contradicts $\mbox{rank}(A)=m$ and $\nu\ne 0$. Similarly, $\Ical^c$ is nonempty in view of $w \ne 0$. Hence we have $(A_{\bullet \Ical})^T w =0$ and $(A_{\bullet \Ical^c})^T \nu =0$. Since $\Ical \cup \Ical^c =\{1, \ldots, N\}$, $\Ical \cap \Ical^c = \emptyset$ and $N \ge 2m-1$, we must have either $|\Ical| \ge m$ or $|\Ical^c| \ge m$. Consider the case of $|\Ical| \ge m$ first. As $(A_{\bullet \Ical^c})^T \nu =0$, we see that $\nu$ is orthogonal to $R(A_{\bullet\Ical^c})$. Since $\nu \in \mathbb R^m$ is nonzero, we obtain $\mbox{rank}(A_{\bullet \Ical^c})<m$. Thus it follows from the properties of $A$ that $\mbox{rank}(A_{\bullet \Ical})=m$, but this contradicts $(A_{\bullet \Ical})^T w=0$ for the nonzero $w$. Using the similar argument, it can be shown that the case of $|\Ical^c| \ge m$ also yields a contradiction. Consequently, $Q$ is positive definite. By the implicit function theorem, there exists a local $C^1$ function $H$ such that $\nu=H(A, y)$ and $F(H(A, y), A, y)=0$ for all $(A, y)$ in a neighborhood of $(A^\diam, y^\diam)$. Let $\nu^\diam:=H(A^\diam, y^\diam)$. Using the chain rule, we have
 \[
    \underbrace{\mathbf J_\nu F(\nu^\diam, A^\diam, y^\diam) }_{:=Q} \cdot \nabla_y H(A^\diam, y^\diam) + \mathbf J_y F(\nu^\diam, A^\diam, y^\diam) = 0.
 \]
Since  $\mathbf J_y F(\nu^\diam, A^\diam, y^\diam) = - I$, we have $\nabla_y H(A^\diam, y^\diam)= Q^{-1}$.

  Observing that $x^*_i=h(a^T_i \nu)$ for each $i=1, \ldots, N$, we deduce via the property of the function $h$ in (\ref{eqn:g_h_def}) that $\mbox{sgn}(x^*_i)=\mbox{sgn}(a^T_i \nu)$ for each $i$. Therefore, in order to show that the zero set of $x^*_i(A, y)$ has zero measure for each $i=1, \ldots, N$, it suffices to show that the zero set of $a^T_i \nu(A, y)$ has zero measure for each $i$. It follows from the previous development that for any $(A^\diam, y^\diam)\in \wt S$, $\nu=H(A, y)$ for a local $C^1$ function $H$ in a neighborhood of $(A^\diam, y^\diam)$. Hence, $\nabla_y  \big( a^T_i \nu \big) (A^\diam, y^\diam) = (a^\diam_i)^T \cdot \nabla_y H(A^\diam, y^\diam) = (a^\diam_i)^T \cdot Q^{-1}$, where $Q:= \mathbf J_\nu F(\nu^\diam, A^\diam, y^\diam)$ is invertible. Since each $a^\diam_i \ne 0$, we have $\nabla_y (a^T_i \nu)(A^\diam, y^\diam) \ne 0$ for each $i=1, \ldots, N$.  In light of Lemma~\ref{lem:measure_of_zero_set}, the zero set of $a^T_i \nu(A, y)$ has zero measure for each $i$. Hence $| \mbox{supp}(x^*(A, y))| = N$ for almost all $(A, y) \in \mathbb R^{m\times N} \times \mathbb R^m$.
\end{proof}

Combining Propositions~\ref{prop:BP_p>2} and \ref{prop:BP_p<2}, we obtain the following result for the generalized basis pursuit.

\begin{theorem} \label{thm:BP_p}
  Let $p>1$ and $N \ge 2m-1$. For almost all $(A, y) \in \mathbb R^{m\times N} \times \mathbb R^m$, the unique optimal solution $x^*_{(A, y)}$ to the BP$_p$ (\ref{eqn:BP_p_norm}) satisfies $|\mbox{supp}(x^*_{(A, y)})| = N$.
\end{theorem}

Motivated by Theorem~\ref{thm:BP_p}, we present the following corollary for a certain fixed measurement matrix $A$ whereas the measurement vector $y$ varies. This result will be used for Theorem~\ref{thm:extension_BP_p} in Section~\ref{sect:comparison}.

\begin{corollary} \label{coro:BP_p_A_fixed}
  Let $p>1$ and $N \ge 2m-1$. Let $A$ be a fixed $m\times N$ matrix such that any of its $m\times m$ submatrix is invertible.  For almost all $y \in \mathbb R^m$, the unique optimal solution $x^*_y$ to the BP$_p$ (\ref{eqn:BP_p_norm}) satisfies $|\mbox{supp}(x^*_{y})| = N$.
\end{corollary}

\begin{proof}
Consider $p \ge 2$ first. For any $y \in \mathbb R^m$, let $x^*(y)$ be the unique optimal solution to the BP$_p$ (\ref{eqn:BP_p_norm}). It follows from the similar argument for Propositions~\ref{prop:BP_p>2} that for  each $i=1, \ldots, N$, $x^*_i$ is a $C^1$ function of $y$ on $\mathbb R^m \setminus \{ 0 \}$, and that  if $x^*_i(y)=0$ for any $y \ne 0$, then the gradient $\nabla_y x^*_i(y) \ne 0$. By Lemma~\ref{lem:measure_of_zero_set}, $| \mbox{supp}(x^*(y))| = N$ for almost all $y \in \mathbb R^m$. When $1< p \le 2$, we note that the given matrix $A$ satisfies the required conditions on $A$ in the set $\wt S$ introduced in the proof of Proposition~\ref{prop:BP_p<2}, since $S$ defined in (\ref{eqn:set_S}) is a proper subset of $\wt S$ as indicated at the end of the first paragraph of the proof of Proposition~\ref{prop:BP_p<2}. Therefore, by the similar argument for  Proposition~\ref{prop:BP_p<2}, we have that for any $y \ne 0$, the gradient $\nabla_y x^*_i(y) \ne 0$. Therefore, the desired result follows.
\end{proof}

%
\subsection{Least Sparsity of  the Generalized Ridge Regression and Generalized Elastic Net with $p>1$}

We first establish the least sparsity of the generalized ridge regression in (\ref{eqn:ridge_reg}) as follows.

\begin{theorem} \label{thm:RR_p}
 Let $p > 1$, $N \ge m$, and $\lambda>0$. For almost all $(A, y) \in \mathbb R^{m\times N} \times \mathbb R^m$, the unique optimal solution $x^*_{(A, y)}$ to the RR$_p$ (\ref{eqn:ridge_reg}) satisfies $| \mbox{supp}(x^*_{(A, y)})| = N$.
\end{theorem}

\begin{proof}
 Recall that for any given $(A, y) \in \mathbb R^{m\times N} \times \mathbb R^m$, the unique optimal solution $x^*$ to the RR$_p$ (\ref{eqn:ridge_reg}) is described by the optimality condition:  $A^T(Ax^*-y) + \lambda \, \nabla f(x^*)=0$, where $\lambda>0$ is a penalty parameter.

 (i) $p>2$. Define the function $F(x, A, y) := \lambda \nabla f(x) + A^T (A x - y)$, where $\nabla f(x) = (g(x_1), \ldots, g(x_N))^T$ with $g$ given in (\ref{eqn:g_h_def}). Hence, the optimal solution $x^*$, as the function of $(A, y)$, satisfies the equation $F(x^*, A, y)=0$. Obviously, $F$ is $C^1$ and its Jacobian with respect to $x$ is given by
 \[
   \mathbf J_x F(x, A, y) \, = \, \lambda D(x) + A^T A,
 \]
 where the diagonal matrix $D(x) := \mbox{diag}( g'(x_1), \ldots, g'(x_N))$.  Since each $g'(x_i)\ge 0$, we see that $\lambda D(x) + A^T A$ is positive semi-definite for all $A$'s and $x$'s.

 We show below that for any $(A^\diam, y^\diam)$ in the set $S$ defined in (\ref{eqn:set_S}), the matrix $\mathbf J_x F(x^\diam, A^\diam, y^\diam)$ is positive definite, where $x^\diam:=x^*(A^\diam, y^\diam)$.
For this purpose, define the index set $\Jcal := \{ i \, | \, x^\diam_i \ne 0 \}$. Partition $D^\diam:=D(x^\diam)$ and $A^\diam$ as $D^\diam= \mbox{diag}( D_1, \,  D_2)$ and $A^\diam = \big [ A_1 \ A_2 \big]$ respectively, where $D_{1} := \mbox{diag}( g'(x^\diam_i) )_{i\in \Jcal^c}=0$, $D_{2} := \mbox{diag}( g'(x^\diam_i) )_{i\in \Jcal}$ is positive definite, $A_1:=A^\diam_{\bullet \Jcal^c}$, and $A_2:=A^\diam_{\bullet \Jcal}$. It follows from Proposition~\ref{prop:sparsity_lower_bound} that $|\Jcal^c| \le m-1$ such that the columns of $A_1$ are linearly independent.
Suppose there exists a vector $z\in \mathbb R^N$ such that $z^T [ \lambda D^\diam + (A^\diam)^T A^\diam ] z =0$. Let $u:=z_{\Jcal^c}$ and $v:=z_\Jcal$. Since $z^T [ \lambda D^\diam + (A^\diam)^T A^\diam ] z \ge z^T \lambda D^\diam z = \lambda v^T D_2 v \ge 0$ and $D_2$ is positive definite, we have $v=0$. Hence, $z^T [\lambda D^\diam + (A^\diam)^T A^\diam] z \ge z^T (A^\diam)^T A^\diam z = \| A^\diam z \|^2_2 = \| A_1 u \|^2_2 \ge 0$. Since the columns of $A_1$ are linearly independent, we have $u=0$ and thus $z=0$. Therefore, $\mathbf J_x F(x^\diam, A^\diam, y^\diam)=\lambda D^\diam + (A^\diam)^T A^\diam$ is positive definite.

By the implicit function theorem, there are local $C^1$ functions $G_1$ and $G_2$ such that $x^*=(x^*_\Jcal, x^*_{\Jcal^c})= (G_1(A, y), G_2(A, y)):= G(A, y)$ and $F(G(A,y), A, y)=0$ for all $(A, y)$ in a neighborhood of $(A^\diam, y^\diam) \in S$. By the chain rule, we have
\[
  \mathbf J_x F(x^\diam, A^\diam, y^\diam) \cdot \begin{bmatrix} \nabla_y G_1(A^\diam, y^\diam) \\ \nabla_y G_2(A^\diam, y^\diam) \end{bmatrix} +  \mathbf J_y F(x^\diam, A^\diam, y^\diam) \, = \, 0,
\]
where $\mathbf J_y F(x^\diam, A^\diam, y^\diam) = - (A^\diam)^T= -[A_1 \ A_2]^T$. Let $P$ be the inverse of $\mathbf J_x F(x^\diam, A^\diam, y^\diam)$, i.e.,
\[
   P = \begin{bmatrix} P_{11} & P_{12} \\ P_{21} & P_{22} \end{bmatrix} \, = \, \begin{bmatrix} \lambda D_1 +  A^T_1 A_1 &  A^T_1 A_2 \\    A^T_2 A_1 & \lambda D_2 +  A^T_2 A_2 \end{bmatrix}^{-1}.
\]
Since $D_1=0$, we obtain  $P_{11} A^T_1 A_1 + P_{12} A^T_2 A_1 = I$. Further, since $\nabla_y G_1(A^\diam, y^\diam) = P_{11} A^T_1 + P_{12} A^T_2$, we have $\nabla_y G_1(A^\diam, y^\diam)\cdot A_1 =I$. This shows that each row of $\nabla_y G_1(A^\diam, y^\diam)$ is nonzero or equivalently the gradient of $x^*_i(A, y)$ is nonzero at $(A^\diam, y^\diam)$ for each $i\in \Jcal^c$. By virtue of Lemma~\ref{lem:measure_of_zero_set}, $| \mbox{supp}(x^*(A, y))| = N$ for almost all $(A, y) \in \mathbb R^{m\times N} \times \mathbb R^m$.


(ii) $1<p\le 2$.  Let $\wh S$ be the set of all $(A, y) \in \mathbb R^{m\times N}\times \mathbb R^m$ such that each column of $A$ is nonzero and $A^T y \ne 0$. Obviously, $\wh S$ is open and its complement has zero measure. Further, for any $(A, y) \in \wh S$, it follows from the optimality condition and $A^T y \ne 0$ that the unique optimal solution $x^* \ne 0$.

Define the function
\[
 F(x, A, y) \, := \, \begin{bmatrix} x_1+ h(\lambda^{-1} a^T_1 (Ax-y)) \\ \vdots \\ x_N + h(\lambda^{-1} a^T_N(Ax-y) ) \end{bmatrix},
\]
  where $h$ is defined in (\ref{eqn:g_h_def}). For any $(A, y) \in \mathbb R^{m\times N}\times \mathbb R^m$, the unique optimal solution $x^*$, as the function of $(A, y)$, satisfies $F(x^*, A, y)=0$. Further, $F$ is $C^1$ and its Jacobian with respect to $x$ is
\[
  \mathbf J_x F(x, A, y) \, = \, I + \lambda^{-1} \cdot \Gamma(x, A, y) A^T A,
\]
where the diagonal matrix $\Gamma(x, A, y)=\mbox{diag}\big( h'(\lambda^{-1} a^T_1 (Ax-y)), \ldots, h'(\lambda^{-1} a^T_N(Ax-y) ) \big)$.

We show next that for any $(A^\diam, y^\diam)\in \wh S$, the matrix $\mathbf J_x F(x^\diam, A^\diam, y^\diam)$ is invertible, where $x^\diam:=x^*(A^\diam, y^\diam)$.
As before, define the index set $\Jcal := \{ i \, | \, x^\diam_i \ne 0 \}$. Since $(A^\diam, y^\diam)\in \wh S$ implies that $x^\diam \ne 0$, the set $\Jcal$ is nonempty.
 Partition $\Gamma^\diam:=\Gamma(x^\diam, A^\diam, y^\diam)$ and $A^\diam$ as $\Gamma^\diam= \mbox{diag}( \Gamma_1, \,  \Gamma_2)$ and $A^\diam = \big [ A_1 \ A_2 \big]$ respectively, where $\Gamma_{1} := \mbox{diag}( h'(\lambda^{-1} (a^\diam_i)^T (A^\diam x^\diam -y^\diam) ) )_{i\in \Jcal^c}=0$, $\Gamma_{2} := \mbox{diag}(  h'(\lambda^{-1}(a^\diam_i)^T (A^\diam x^\diam -y^\diam) ) )_{i\in \Jcal}$ is positive definite, $a^\diam_i$ is the $i$th column of $A^\diam$, $A_1:=A^\diam_{\bullet \Jcal^c}$, and $A_2:=A^\diam_{\bullet \Jcal}$. Therefore, we obtain
\begin{equation} \label{eqn:Jacobian_RR}
 \mathbf J_x F(x^\diamond, A^\diam, y^\diam) \, = \, \begin{bmatrix} I & 0 \\  \lambda^{-1} \Gamma_2 A^T_2 A_{1} & I + \lambda^{-1} \Gamma_2 A^T_2 A_{2} \end{bmatrix}.
\end{equation}
Since $\Gamma_2$ is positive definite, we deduce that $I + \lambda^{-1} \Gamma_2 A^T_2 A_{2}= \Gamma_2 ( \Gamma^{-1}_2 + \lambda^{-1} A^T_2 A_2)$ is invertible. Hence  $\mathbf J_x F(x^\diamond, A^\diam, y^\diam)$ is invertible.
By the implicit function theorem, there are local $C^1$ functions $G_1$ and $G_2$ such that $x^*=(x^*_\Jcal, x^*_{\Jcal^c})= (G_1(A, y), G_2(A, y)):= G(A, y)$ and $F(G(A,y), A, y)=0$ for all $(A, y)$ in a neighborhood of $(A^\diam, y^\diam)\in\mathbb R^{m\times N}\times \mathbb R^m$. By the chain rule, we have
\[
   \mathbf J_x F(x^\diamond, A^\diam, y^\diam)  \cdot \begin{bmatrix} \nabla_y G_1(A^\diam, y^\diam) \\ \nabla_y  G_2(A^\diam, y^\diam) \end{bmatrix} \, = \, - \mathbf J_y F(x^\diamond, A^\diam, y^\diam) \, = \, \begin{bmatrix}  \Gamma_1 \,  \lambda^{-1} A^T_1 \\ \Gamma_2 \,  \lambda^{-1} A^T_2  \end{bmatrix} \in \mathbb R^{N\times m}.
\]
 In view of (\ref{eqn:Jacobian_RR}), $\Gamma_1=0$, and the invertibility of $\mathbf J_x F(x^\diamond, A^\diam, y^\diam)$, we obtain $\nabla_y G_1(A^\diam, y^\diam)=0$ and $\nabla_y G_2(A^\diam, y^\diam)= ( I + \lambda^{-1} \Gamma_2 A^T_2 A_{2})^{-1} \Gamma_2  \lambda^{-1} A^T_2 $.

 Noting that $x^*_i= - h(\lambda^{-1} a^T_i (A x^* - y))$ for each $i=1, \ldots, N$, we deduce via the property of the function $h$ in (\ref{eqn:g_h_def}) that $\mbox{sgn}(x^*_i)=\mbox{sgn}( a^T_i (y- A x^*))$ for each $i$. Therefore, it suffices to show that the zero set of $a^T_i (y-A x^*)$ has zero measure for each $i=1, \ldots, N$.
  It follows from the previous development that for any $(A^\diam, y^\diam) \in \wh S$, $(x^*_{\Jcal^c}, x^*_\Jcal)=(G_1(A, y), G_2(A, y))$ in a neighborhood of $(A^\diam, y^\diam)$ for local $C^1$ functions $G_1$ and $G_2$.
 %
 %
  For each $i\in \Jcal^c$, define
  \[
    q_i(A, y) \, := \, a^T_i \big(y- A \cdot x^*(A, y) \big).
  \]
     Then $\nabla_y \, q_i(A^\diam, y^\diam)=(a^\diam_i)^T (I-A_2 \cdot \nabla_y G_2(A^\diam, y^\diam))$. Note that by the Sherman-Morrison-Woodbury formula \cite[Section 3.8]{Meyer_book00}, we have
  \[
    A_2 \cdot \nabla_y G_2(A^\diam, y^\diam) \, = \,  A_2 ( I + \lambda^{-1} \Gamma_2 A^T_2 A_{2})^{-1} \Gamma_2  \lambda^{-1} A^T_2 \, = \, I - (I+ \lambda^{-1} A_2 \Gamma_2 A^T_2)^{-1},
  \]
 where it is easy to see that $I+ \lambda^{-1} A_2 \Gamma_2 A^T_2$ is invertible.
 Hence, for any $(A^\diam, y^\diam) \in \wh S$, we deduce via $a^\diam_i \ne 0$ that $\nabla_y \, q_i(A^\diam, y^\diam) =  (a^\diam_i)^T (I+ \lambda^{-1} A_2 \Gamma_2 A^T_2)^{-1} \ne 0$ for each $i \in \Jcal^c$. In view of Lemma~\ref{lem:measure_of_zero_set}, $| \mbox{supp}(x^*(A, y))| = N$ for almost all $(A, y) \in \mathbb R^{m\times N} \times \mathbb R^m$.
\end{proof}

The next result pertains to the generalized elastic net (\ref{eqn:elastic_net}).

\begin{theorem} \label{thm:EN_p}
 Let $p > 1$, $N \ge m$, $r \ge 1$, and $\lambda_1, \lambda_2>0$. For almost all $(A, y) \in \mathbb R^{m\times N} \times \mathbb R^m$, the unique optimal solution $x^*_{(A, y)}$ to the EN$_p$ (\ref{eqn:elastic_net})
 satisfies $| \mbox{supp}(x^*_{(A, y)})| = N$.
\end{theorem}

\begin{proof}
 Recall that for any given $(A, y) \in \mathbb R^{m\times N} \times \mathbb R^m$, the unique optimal solution $x^*$ to the EN$_p$ (\ref{eqn:elastic_net}) is characterized by equation (\ref{eqn:EN_optimality}):
\begin{equation} \label{eqn:ENp_opt}
   A^T(Ax^*-y) + p^{-1} r \lambda_1 \cdot \| x^* \|^{r-p}_p \cdot \nabla f(x^*) + 2 \lambda_2 \, x^* =0,
\end{equation}
  where $r \ge 1$, and $\lambda_1, \lambda_2>0$ are the penalty parameters.

(i) $p\ge 2$.  Consider the open set $S$ defined in (\ref{eqn:set_S}). For any $(A, y)\in S$, since $A$ has full row rank and $y \ne 0$, we have $A^T y \ne 0$.
%
%
Hence, it follows from the optimality condition (\ref{eqn:ENp_opt}) and $A^T y \ne 0$ that the unique optimal solution $x^* \ne 0$ for any $(A, y) \in S$.

Define the function $F(x, A, y) :=  A^T(Ax-y) + p^{-1} r \lambda_1 \cdot \| x \|^{r-p}_p \cdot \nabla f(x) + 2 \lambda_2 \, x$, where $\nabla f(x) = (g(x_1), \ldots, g(x_N))^T$. Hence, the optimal solution $x^*$, as the function of $(A, y)$, satisfies the equation $F(x^*, A, y)=0$. Since $\| \cdot \|_p$ is $C^2$ on $\mathbb R^N\setminus \{0\}$, we see that $F$ is $C^1$ on the open set $(\mathbb R^N\setminus \{ 0\}) \times \mathbb R^{m\times N} \times \mathbb R^m$, and its Jacobian with respect to $x$ is given by
 \[
   \mathbf J_x F(x, A, y) \, = \,  A^T A + \lambda_1 \mathbf H(\| x \|^r_p) + 2 \lambda_2 I,
 \]
 where $\mathbf H(\| x \|^r_p)$ denotes the Hessian of $\| \cdot \|^r_p$ at any nonzero $x$. Since $r \ge 1$, $\| \cdot \|^r_p$ is a convex function and its Hessian at any nonzero $x$ must be positive semi-definite. This shows that for any $(A^\diam, y^\diam) \in S$,  $\mathbf J_x F(x^\diam, A^\diam, y^\diam)$ is positive definite, where $x^\diam:=x^*(A^\diam, y^\diam)\ne 0$. Hence, there exists a local $C^1$ function $G$ such that $x^*=G(A, y)$ with $F(G(A, y), A, y)=0$ for all $(A, y)$ in a neighborhood of $(A^\diam, y^\diam)$.

  For any given $(A^\diam, y^\diam)\in S$, define the (nonempty) index set $\Jcal := \{ i \, | \, x^\diam_i \ne 0 \}$. Let $\Lambda(x):=\diag(g'(x_1), \ldots, g'(x_N))$. Partition $\Lambda^\diam:=\Lambda(x^\diam)$ and $A^\diam$ as $\Lambda^\diam= \mbox{diag}( \Lambda_1, \,  \Lambda_2)$ and $A^\diam = \big [ A_1 \ A_2 \big]$ respectively, where $\Lambda_{1} := \mbox{diag}( g'(x^\diam_i) )_{i\in \Jcal^c}=0$, $\Lambda_{2} := \mbox{diag}( g'(x^\diam_i) )_{i\in \Jcal}$ is positive definite, $A_1:=A^\diam_{\bullet \Jcal^c}$, and $A_2:=A^\diam_{\bullet \Jcal}$. Using  $\nabla ( \| x \|_p ) = (p \| x \|^{p-1}_p)^{-1} \cdot \nabla (\| x \|^p_p)$ for any $x \ne 0$, we have, for any $x \ne 0$,
  \[
    \mathbf H(\|x\|^r_p) \, = \, \frac{ r}{p} \cdot \| x \|^{r-p}_p \cdot \Big[ \Lambda(x) + \frac{r-p}{p \| x \|^p_p} \cdot \nabla f(x) \big( \nabla f(x) \big)^T \Big].
  \]
  Based on the partition given above, we have $\mathbf H(\|x^\diam\|^{r-p}_p) = \diag(0, H_2)$, where the matrix $H_2$ is positive semi-definite. Therefore, we obtain
  \[
    \mathbf J_x F(x^\diam, A^\diam, y^\diam) = \begin{bmatrix}  A^T_1 A_1 + 2  \lambda_2 I &  A^T_1 A_2 \\  A^T_2 A_1 &  A^T_2 A_2 + 2  \lambda_2 I + \lambda_1 H_2 \end{bmatrix}, \quad \mathbf J_y F(x^\diam, A^\diam, y^\diam) = -\begin{bmatrix}  A^T_1 \\  A^T_2 \end{bmatrix}.
  \]
  Let $Q$ be the inverse of $\mathbf J_x F(x^\diam, A^\diam, y^\diam)$, i.e.,
  $
     Q \, = \, \begin{bmatrix} Q_{11} & Q_{12} \\ Q_{21} & Q_{22} \end{bmatrix}.
  $
  Hence,  we have
 %
 %
 \begin{equation} \label{eqn:ENp_Q}
 (Q_{11} A^T_1 + Q_{12} A^T_2) A_2 + Q_{12} \big(2\lambda_2 I+ \lambda_1 H_2 \big)=0.
 \end{equation}
  We claim that each row of $Q_{11} A^T_1 + Q_{12} A^T_2$ is nonzero. Suppose not, i.e., $(Q_{11} A^T_1 + Q_{12} A^T_2)_{i\bullet}=0$ for some $i$. Then it follows from (\ref{eqn:ENp_Q}) that $(Q_{12})_{i\bullet}(2\lambda_2 I+\lambda_1 H_2)=0$. Since $2\lambda_2 I+\lambda_1 H_2$ is positive definite, we have $(Q_{12})_{i\bullet}=0$. By $(Q_{11} A^T_1 + Q_{12} A^T_2)_{i\bullet}=0$, we obtain $(Q_{11})_{i\bullet} A^T_1=0$.  It follows from Proposition~\ref{prop:sparsity_lower_bound} that $|\Jcal^c| \le m-1$ such that the columns of $A_1$ are linearly independent.
%
%
  Hence,  we have $(Q_{11})_{i\bullet}=0$ or equivalently $Q_{j\bullet}=0$ for some $j$. This contradicts the invertibility of $Q$, and thus completes the proof of the claim. Furthermore, let $G_1, G_2$ be local $C^1$ functions such that $x^*=(x^*_{\Jcal^c}, x^*_\Jcal)=(G_1(A, y), G_2(A, y))$ for all $(A, y)$ in a neighborhood of $(A^\diam, y^\diam) \in S$.
   By the similar argument as before, we see that  $\nabla_y G_1(A^\diam, y^\diam)= Q_{11} A^T_1 + Q_{12} A^T_2$. Therefore, we deduce that the gradient of $x^*_i(A, y)$ at $(A^\diam, y^\diam)$ is nonzero for each $i\in \Jcal^c$.  In light of Lemma~\ref{lem:measure_of_zero_set}, $| \mbox{supp}(x^*(A, y))| = N$ for almost all $(A, y) \in \mathbb R^{m\times N} \times \mathbb R^m$.


 (ii) $1<p\le 2$. Let $\wh S$ be the set defined in Case (ii) of Theorem~\ref{thm:RR_p}, i.e., $\wh S$ is the set of all $(A, y) \in \mathbb R^{m\times N}\times \mathbb R^m$ such that each column of $A$ is nonzero and $A^T y \ne 0$. The set $\wh S$ is open and its complement has zero measure. For any $(A, y) \in \wh S$, the unique optimal solution $x^*$, treated as a function of $(A, y)$, is nonzero.
%
%
%
 %
By the optimality condition (\ref{eqn:ENp_opt}) and the definition of the function $h$ in (\ref{eqn:g_h_def}), we see that $x^*$ satisfies the following equation for any $(A, y) \in \wh S$:
\begin{equation} \label{eqn:EN_F}
   F(x, A, y) \, := \, \begin{bmatrix} x_1 + h(w_1) \\ \vdots \\ x_N + h(w_N) \end{bmatrix} = 0,
\end{equation}
where $w_i \, := \, p \| x\|^{p-r}_p \cdot \big[ a^T_i(Ax -y) + 2 \lambda_2 x_i \big]/(r\lambda_1 )$ for each $i=1, \ldots, N$. It is easy to show that $F$ is $C^1$ on $(\mathbb R^N\setminus \{ 0\}) \times \mathbb R^{m\times N} \times \mathbb R^m$ and its Jacobian with respect to $x$ is
\[
  \mathbf J_x F(x, A, y) \, = \, I + \frac{p}{r \lambda_1} \cdot  \Gamma(x, A, y) \cdot \Big\{ [A^T(Ax -y) + 2 \lambda_2 x ]\cdot [\nabla( \| x \|^{p-r}_r )]^T + \| x \|^{p-r}_p (A^T A + 2 \lambda_2 I) \Big\},
\]
where the diagonal matrix $\Gamma(x, A, y):=\mbox{diag}\big( h'(w_1), \ldots, h'(w_N) \big)$ and $\nabla( \| x \|^{p-r}_r )= (p-r) \nabla f(x)/[ p \cdot \| x \|^r_p]$.

We show next that  $\mathbf J_x F(x^\diam, A^\diam, y^\diam)$ is invertible for any $(A^\diam, y^\diam)\in \wh S$, where $x^\diam:=x^*(A^\diam, y^\diam)$.
As before, define the (nonempty) index set $\Jcal := \{ i \, | \, x^\diam_i \ne 0 \}$.
 Partition $\Gamma^\diam:=\Gamma(x^\diam, A^\diam, y^\diam)$ and $A^\diam$ as $\Gamma^\diam= \mbox{diag}( \Gamma_1, \,  \Gamma_2)$ and $A^\diam = \big [ A_1 \ A_2 \big]$ respectively, where $\Gamma_{1} := \mbox{diag}( h'( w_i) )_{i\in \Jcal^c}=0$, $\Gamma_{2} := \mbox{diag}(  h'(w_i) )_{i\in \Jcal}$ is positive definite, $a^\diam_i$ is the $i$th column of $A^\diam$, $A_1:=A^\diam_{\bullet \Jcal^c}$, and $A_2:=A^\diam_{\bullet \Jcal}$. Therefore, we obtain
\begin{equation*} 
 W \, := \, \mathbf J_x F(x^\diamond, A^\diam, y^\diam) \, = \, \begin{bmatrix} I & 0 \\  \star &  W_{22}\end{bmatrix},
\end{equation*}
where  by letting the vector $\wt b:= (\nabla f(x^\diam))_\Jcal$,
\[
  W_{22} \, := \,  I + \frac{p}{r \lambda_1} \cdot \Gamma_2 \cdot \Big\{ [A^T_2(A_2 x^\diam_\Jcal -y) + 2 \lambda_2 x^\diam_\Jcal ]\cdot \frac{(p-r) \wt b^T}{p \| x^\diam\|^r_p}  + \| x^\diam \|^{p-r}_p (A^T_2 A_2 + 2 \lambda_2 I) \Big\}.
\]
It follows from (\ref{eqn:ENp_opt}) that $\frac{p}{r \lambda_1} \cdot [A^T_2(A_2 x^\diam_\Jcal -y) + 2 \lambda_2 x^\diam_\Jcal] = - \| x^\diam\|^{r-p}_p \cdot \wt b$. Hence,
\begin{equation} \label{eqn:Gamma_inv}
  \Gamma^{-1}_2 \cdot W_{22} \, = \, \underbrace{\Gamma^{-1}_2 + \frac{r-p}{p\|x^\diam\|^p_p} \cdot \wt b \cdot \wt b^T }_{:=U} \,  + \,  \frac{p \| x^\diam \|^{p-r}_p}{r \lambda_1} \cdot  (A^T_2 A_2 + 2 \lambda_2 I).
\end{equation}
Clearly, when $r\ge p>1$, the matrix $\Gamma^{-1}_2 W_{22}$ is positive definite. In what follows, we consider the case where $2\ge p>r\ge 1$. Let  the vector $b:=( \sgn(x^\diam_i)\cdot |x^\diam_i|^{p-1} )_{i\in \Jcal}$ so that $\wt b = p \cdot b$. In view of (\ref{eqn:EN_F}), we have $w_i = h^{-1}(-x^\diam_i)= p \cdot \mbox{sgn}(-x_i) |x^\diam_i|^{p-1}$ for each $i$. Therefore, using the formula for $h'(\cdot)$ given below (\ref{eqn:g_h_def}), we obtain that for each $i\in \mathcal J$,
\[
    h'(w_i) \, = \, \frac{|w_i|^{\frac{2-p}{p-1}} }{(p-1)\cdot p^{\frac{1}{p-1}} } \, = \, \frac{ \big( p \cdot |x^\diam_i|^{p-1} \big)^{\frac{2-p}{p-1}} } { (p-1)\cdot p^{\frac{1}{p-1}} } \, = \, \frac{ |x^\diam_i|^{2-p} }{(p-1) \cdot p}.
\]
This implies that  $\Gamma^{-1}_2 = p(p-1) D$, where the diagonal matrix $D:=\diag( |x^\diam_i|^{p-2} )_{i\in \Jcal}$. Clearly, $D$ is positive definite.
%
%
 We thus have, via $p-1\ge p-r >0$,
\[
     U \, = \, \Gamma^{-1}_2 + \frac{r-p}{p\|x^\diam\|^p_p} \cdot \wt b \cdot \wt b^T = p (p-1) \left( D - \frac{p-r}{p-1} \cdot \frac{b \cdot b^T}{ \| x\|^p_p } \right) \, \succcurlyeq \,  p (p-1) \left( D - \frac{b \cdot b^T}{ \| x\|^p_p } \right),
\]
where $\succcurlyeq$ denotes the positive semi-definite order. Since the diagonal matrix $D$ is positive definite, we further have
 \[
   D - \frac{b \cdot b^T}{\| x \|^p_p} \, = \, D^{1/2} \Big( I - \frac{D^{-1/2} b \cdot b^T D^{-1/2}}{\| x\|^p_p} \Big) D^{1/2} \, = \,  D^{1/2} \Big( I - \frac{u \cdot u^T}{ \| u \|^2_2} \Big) D^{1/2},
 \]
 where $u := D^{-1/2} \cdot b=\big(\mbox{sgn}(x^\diam_i)|x^\diam_i|^{p/2}\big)_{i\in \Jcal} $ such that $\| u\|^2_2= \| x \|^p_p$. Since $I- \frac{u \cdot u^T}{ \| u \|^2_2}$ is positive semi-definite, so is $D - \frac{b \cdot b^T}{\| x \|^p_p}$. This shows that $U$ in (\ref{eqn:Gamma_inv}) is positive semi-definite. Since the last term on the right hand side of (\ref{eqn:Gamma_inv}) is positive definite, $\Gamma^{-1}_2 W_{22}$ is positive definite. Therefore, $W_{22}$ is invertible, and so is $W$ for all $1<p\le 2$ and $r \ge 1$.
By the implicit function theorem, there are local $C^1$ functions $G_1$ and $G_2$ such that $x^*=(x^*_\Jcal, x^*_{\Jcal^c})= (G_1(A, y), G_2(A, y)):= G(A, y)$ and $F(G(A,y), A, y)=0$ for all $(A, y)$ in a neighborhood of $(A^\diam, y^\diam)\in\mathbb R^{m\times N}\times \mathbb R^m$. Moreover, we have
\[
   \mathbf J_x F(x^\diamond, A^\diam, y^\diam)  \cdot \begin{bmatrix} \nabla_y G_1(A^\diam, y^\diam) \\ \nabla_y  G_2(A^\diam, y^\diam) \end{bmatrix} \, = \,- \mathbf J_y F(x^\diamond, A^\diam, y^\diam) \, = \, \frac{p \| x^\diam\|^{p-r}_p }{ r \lambda_1}
   \begin{bmatrix} \Gamma_1 \,  A^T_1 \\  \Gamma_2 \,  A^T_2  \end{bmatrix} \in \mathbb R^{N\times m}.
\]
%
%
 In view of the invertibility of $\mathbf J_x F(x^\diamond, A^\diam, y^\diam)$  and $\Gamma_1=0$, we obtain
 \[
 \nabla_y G_1(A^\diam, y^\diam)=0, \qquad \nabla_y G_2(A^\diam, y^\diam)= \left(\frac{p \| x^\diam\|^{p-r}_p }{ r \lambda_1} \right) W^{-1}_{22} \Gamma_2  A^T_2.
\]

Since $x^*_i= - h(w_i)$ for each $i=1, \ldots, N$, where $w_i$ is defined below (\ref{eqn:EN_F}), we deduce via the positivity of $\|x\|_p$ and the property of the function $h$  in (\ref{eqn:g_h_def}) that $\mbox{sgn}(x^*_i)=\mbox{sgn}( a^T_i(y- A x^*) - 2 \lambda_2 x^*_i)$ for each $i$.
 For each $i\in \Jcal^c$, define
  \[
    q_i(A, y) \, := \, a^T_i \big (y - A \cdot x^*(A, y) -y \big) - 2 \lambda_2 x^*_i(A, y).
  \]
In what follows, we show that for each $i\in \Jcal^c$, the gradient of $q_i(A, y)$ at $(A^\diam, y^\diam)\in \wh S$ is nonzero.
  It follows from the previous development that for any $(A^\diam, y^\diam) \in \wh S$, $(x^*_{\Jcal^c}, x^*_\Jcal)=(G_1(A, y), G_2(A, y))$ in a neighborhood of $(A^\diam, y^\diam)$ for local $C^1$ functions $G_1$ and $G_2$.
 Using $\nabla_y G_1(A^\diam, y^\diam)=0$, we have $\nabla_y \, q_i(A^\diam, y^\diam)=(a^\diam_i)^T (I- A_2 \cdot \nabla_y G_2(A^\diam, y^\diam))$. Letting $\alpha:=\frac{p \| x^\diam\|^{p-r}_p }{ r \lambda_1}>0$ and by (\ref{eqn:Gamma_inv}), we have
  \[
    A_2 \cdot \nabla_y G_2(A^\diam, y^\diam) \, = \, \alpha A_2  W^{-1}_{22} \Gamma_2  A^T_2 = \alpha A_2 (\Gamma^{-1}_2 W_{22} )^{-1} A^T_2 \, = \, \alpha A_2 \big[ U+ \alpha (A^T_2 A_2 + 2\lambda_2 I) \big]^{-1} A^T_2.
  \]
  Since $U$ is positive semi-definite and $A^T_2 A_2 + 2\lambda_2 I$ is positive definite, it can be shown that
  \[
    A_2  \big(A^T_2 A_2 + 2\lambda_2 I \big)^{-1} A^T_2 \, \succcurlyeq \, \alpha A_2 \big[ U + \alpha (A^T_2 A_2 + 2\lambda_2 I)  \big ]^{-1} A^T_2 \, \succcurlyeq \, 0.
  \]
  Since each eigenvalue of $A_2 (A^T_2 A_2 + 2\lambda_2 I)^{-1} A^T_2$ is strictly less than one, we conclude that $A_2 \cdot \nabla_y G_2(A^\diam, y^\diam)$ is positive semi-definite and each of its eigenvalues is strictly less than one. Therefore, $I-A_2 \cdot \nabla_y G_2(A^\diam, y^\diam)$ is invertible. Since each $a^\diam_i \ne 0$, we have $\nabla_y \, q_i(A^\diam, y^\diam) \ne 0$ for each $i \in \Jcal^c$.
 In view of Lemma~\ref{lem:measure_of_zero_set}, $| \mbox{supp}(x^*(A, y))| = N$ for almost all $(A, y) \in \mathbb R^{m\times N} \times \mathbb R^m$.
%
%
\end{proof}

%
\subsection{Least Sparsity of the Generalized Basis Pursuit Denoising  with $p>1$} 

We consider the cases where $p \ge 2$ first.

\begin{proposition} \label{prop:full_support_BPDN_P>2}
 Let $p \ge 2$ and $N \ge m$. For almost all $(A, y) \in \mathbb R^{m\times N} \times \mathbb R^m$ with $y \ne 0$, if  $0<\varepsilon<\| y \|_2$, then
  the unique optimal solution $x^*_{(A, y)}$ to the BPDN$_p$ (\ref{eqn:BP_denoising01}) satisfies $|\mbox{supp}(x^*_{(A, y)})| = N$.
\end{proposition}

\begin{proof}
  Consider the set $S$ defined in (\ref{eqn:set_S}). It follows from the proof for Case (ii) in Proposition~\ref{prop:sparsity_lower_bound} that for any given $(A, y) \in S$ and any $\varepsilon>0$ with $\varepsilon<\| y \|_2$, the unique optimal solution $x^*$ satisfies the optimality conditions: $\nabla f(x^*) + 2 \mu A^T (A x^* - y)  = 0$ for a unique positive $\mu$, and $\| A x^* - y \|^2_2 = \varepsilon^2$. Hence, $(x^*, \mu) \in \mathbb R^{N+1}$ is a function of $(A, y)$ on $S$ and satisfies the following equation:
 \[
 F(x, \mu, A, y) \, := \,  \begin{bmatrix} g(x_{1}) + 2 \mu \, a^T_1 (A x - y) \\ \vdots \\ g(x_N) + 2 \mu \, a^T_N (A x - y)  \\ \| A x - y \|^2_2 -\varepsilon^2 \end{bmatrix} \, = \, 0.
 %
 \]
 Clearly, $F:\mathbb R^N \times \mathbb R \times \mathbb R^{m\times N} \times \mathbb R^m \rightarrow \mathbb R^{N+1}$ is $C^1$ and its Jacobian with respect to $(x, \mu)$ is given by
 \[
   \mathbf J_{(x, \mu)} F(x, \mu, A, y) \, = \, \begin{bmatrix} M(x, \mu, A) & 2A^T (A x - y) \\ 2 (Ax - y)^T A & 0 \end{bmatrix},
 \]
 where $M(x, \mu, A) : =  \Lambda(x) + 2\mu A^T A$, and the diagonal matrix $\Lambda(x):=\diag(g'(x_1), \ldots, g'(x_N))$ is positive semi-definite.
 Given $(A^\diam, y^\diam) \in S$, define $x^\diam:=x^*(A^\diam, y^\diam)$ and $\mu^\diam:=\mu(A^\diam, y^\diam)>0$. We claim that $\mathbf J_{(x, \mu)} F(x^\diam, \mu^\diam, A^\diam, y^\diam)$ is invertible for any $(A^\diam, y^\diam) \in S$. To show it, define the index set $\Jcal := \{ i \, | \, x^\diam_i \ne 0 \}$. Note that $\mathcal J$ is nonempty by virtue of Proposition~\ref{prop:sparsity_lower_bound}.
  Partition $\Lambda^\diam:=\Lambda(x^\diam)$ and $A$ as $\Lambda^\diam= \mbox{diag}( \Lambda_1, \,  \Lambda_2)$ and $A^\diam = \big [ A_1 \ A_2 \big]$ respectively, where $\Lambda_{1} := \mbox{diag}( g'(x^\diam_i) )_{i\in \Jcal^c}=0$, $\Lambda_{2} := \mbox{diag}( g'(x^\diam_i) )_{i\in \Jcal}$ is positive definite, $A_1:=A^\diam_{\bullet \Jcal^c}$, and $A_2:=A^\diam_{\bullet \Jcal}$.
 It follows from the similar argument for Case (i) in Theorem~\ref{thm:RR_p} that  $M^\diam:=M(x^\diam, \mu^\diam, A^\diam)$ is positive definite.
Moreover, it has been shown in the proof of Proposition~\ref{prop:sparsity_lower_bound} that $b:=2(A^\diamond)^T (A^\diamond x^\diamond - y^\diamond) \in \mathbb R^N$ is nonzero. Hence, for any $z=[z_1; z_2]\in \mathbb R^{N+1}$ with $z_1 \in \mathbb R^N$ and $z_2 \in \mathbb R$, we have
\[
  \mathbf J_{(x, \mu)} F(x^\diam, \mu^\diam, A^\diam, y^\diam) z = \begin{bmatrix} M^\diam & b \\ b^T & 0 \end{bmatrix} \begin{pmatrix} z_1 \\ z_2 \end{pmatrix} = 0 \ \Rightarrow \ M^\diam z_1 + b z_2 =0, \ b^T z_1 = 0 \ \Rightarrow \ b^T \big(M^\diam \big)^{-1} b z_2 =0.
\]
This implies that $z_2=0$ and further $z_1=0$. Therefore, $\mathbf J_{(x, \mu)} F(x^\diam, \mu^\diam, A^\diam, y^\diam)$ is invertible.
%
%
By the implicit function theorem, there are local $C^1$ functions $G_1, G_2, H$ such that $x^*=(x^*_{\Jcal^c}, x^*_\Jcal)=(G_1(A, y), G_2(A, y)):=G(A, y)$, $\mu=H(A, y)$, and $F(G(A, y), H(A, y), A, y)=0$ for all $(A, y)$ in a neighborhood of $(A^\diam, y^\diam)$.
  By the chain rule, we have
\[
   \underbrace{\mathbf J_{(x, \mu)} F(x^\diam, \mu^\diam, A^\diam, y^\diam)}_{:=V} \cdot \begin{bmatrix} \nabla_y G_1(A^\diam, y^\diam) \\ \nabla_y G_2(A^\diam, y^\diam) \\ \nabla_y H(A^\diam, y^\diam) \end{bmatrix} \, = \, - \mathbf J_{y} F(x^\diam, \mu^\diam, A^\diam, y^\diam) \, = \, \begin{bmatrix} 2 \mu^\diam A^T_1 \\ 2 \mu^\diam A^T_2 \\  2 (A^\diam x^\diam - y^\diam)^T \end{bmatrix},
\]
where
by using $\Lambda_{1}=0$,
\[
    V = \begin{bmatrix} 2 \mu^\diamond A^T_1 A_1 & 2 \mu^\diam  A^T_1 A_2 &  2 A^T_1 (A^\diam x^\diamond - y^\diamond) \\ 2 \mu^\diam A^T_2 A_1 & \Lambda_2 + 2 \mu^\diam  A^T_2 A_2 &  2 A^T_2 (A^\diam x^\diamond - y^\diamond) \\ 2 (A^\diam x^\diamond - y^\diamond)^T A_1  & 2  (A^\diam x^\diamond - y^\diamond)^T A_2 & 0 \end{bmatrix}. 
%
\]
Let $P$ be the inverse of $V$ defined by
\[
  P =\begin{bmatrix} P_{11} & P_{12} & P_{13} \\ P_{21} & P_{22} & P_{23} \\ P_{31} & P_{32} & P_{33} \end{bmatrix}.
\]
Hence, we have
$
  P_{11}  2 \mu^\diamond A^T_1 A_1 + P_{12} \, 2 \mu^\diam A^T_2 A_1 + P_{13} \, 2 (A^\diam x^\diamond - y^\diamond)^T A_1 = I_m.
$
Since
\[
  \nabla_y G_1(A^\diam, y^\diam) \, = \, -[ P_{11} \ P_{12} \ P_{13} ] \cdot \mathbf J_{y} F(x^\diam, \mu^\diam, A^\diam, y^\diam) = P_{11}  2 \mu^\diamond A^T_1  + P_{12} \, 2 \mu^\diam A^T_2 + P_{13} \, 2 (A^\diam x^\diamond - y^\diamond)^T,
\]
   we have $\nabla_y G_1(A^\diam, y^\diam) \cdot A_1 = I_m$. This shows that each row of $\nabla_y G_1(A^\diam, y^\diam)$ is nonzero. Hence, for each $i\in \Jcal^c$, the gradient of $x^*_i(A, y)$ at $(A^\diam, y^\diam)\in S$ is nonzero.
Consequently, by Lemma~\ref{lem:measure_of_zero_set},
%
%
$| \mbox{supp}(x^*(A, y))| = N$ for almost all $(A, y) \in \mathbb R^{m\times N} \times \mathbb R^m$.
\end{proof}

In the next, we consider the case where $1<p\le 2$.

\begin{proposition} \label{prop:BPDN_p<2}
 Let $1<p \le 2$ and $N \ge m$. For almost all $(A, y) \in \mathbb R^{m\times N} \times \mathbb R^m$ with $y \ne 0$, if $0<\varepsilon<\| y \|_2$, then the unique optimal solution $x^*_{(A, y)}$ to the BPDN$_p$ (\ref{eqn:BP_denoising01}) satisfies $| \mbox{supp}(x^*_{(A, y)})| = N$.
\end{proposition}

\begin{proof}
    Let $\breve S$  be the set of all $(A, y) \in \mathbb R^{m\times N}\times \mathbb R^m$ such that each column of $A$ is nonzero and $y \ne 0$. Obviously, $\breve S$ is open in $\mathbb R^{m\times N}\times \mathbb R^m$ and its complement has zero measure.
  For any $(A, y) \in \breve S$ and any positive $\varepsilon$ with $\varepsilon<\| y\|_2$, it follows from the proof for Case (ii) in Proposition~\ref{prop:sparsity_lower_bound} that the unique optimal solution $x^* \ne 0$ with a unique positive $\mu$. Further, $(x^*, \mu) \in \mathbb R^{N+1}$ satisfies the following equation:
 \[
 F(x, \mu, A, y) \, := \,  \begin{bmatrix} x_{1} +  h(2 \mu a^T_1 (A x - y)) \\ \vdots \\ x_N + h(2 \mu a^T_N (A x - y)) \\ \| A x - y \|^2_2 -\varepsilon^2 \end{bmatrix} \, = \, 0,
%
%
 \]
 where $h$ is defined in (\ref{eqn:g_h_def}). Hence, $F:\mathbb R^N \times \mathbb R \times \mathbb R^{m\times N} \times \mathbb R^m \rightarrow \mathbb R^{N+1}$ is $C^1$ and its Jacobian with respect to $(x, \mu)$ is given by
 \[
   \mathbf J_{(x, \mu)} F(x, \mu, A, y) \, = \, \begin{bmatrix} V(x, \mu, A, y) & 2 \Gamma(x, \mu, A, y) A^T (A x - y) \\ 2 (Ax - y)^T A & 0 \end{bmatrix} \in \mathbb R^{(N+1)\times (N+1)},
 \]
 where  $\Gamma(x, \mu, A, y):=\mbox{diag}\big( h'(2 \mu a^T_1 (A x - y)), \ldots, h'(2 \mu a^T_N (A x - y))  \big) \in \mathbb R^{N\times N}$, and $V(x, \mu, A, y) := I + \Gamma(x, \mu, A, y) \, 2 \mu A^T A$.

We use the same notation $x^\diam$ and $\mu^\diam$ as before. For any $(A^\diamond, y^\diamond) \in \mathbb R^{m\times N} \times \mathbb R^m$, define the index set $\Jcal := \{ i \, | \, x^\diam_i \ne 0 \}$.
Note that $\Jcal$ is nonempty as $\| y \|_2> \varepsilon$.
  Partition $\Gamma^\diamond:= \Gamma(x^\diamond, \mu^\diamond, A^\diamond, y^\diamond)$ and $A^\diam$ as $\Gamma^\diam= \mbox{diag}( \Gamma_1, \,  \Gamma_2)$ and $A^\diam = \big [ A_1 \ A_2 \big]$ respectively, where $\Gamma_{1} := \mbox{diag}( h'(2 \mu^\diam (a^\diam_i)^T (A^\diam x^\diam - y^\diam)) )_{i\in \Jcal^c}=0$, $\Gamma_{2} := \mbox{diag}( h'(2 \mu^\diam (a^\diam_i)^T (A^\diam x^\diam - y^\diam)) )_{i\in \Jcal}$ is positive definite, $A_1:=A^\diam_{\bullet \Jcal^c}$, and $A_2:=A^\diam_{\bullet \Jcal}$. Therefore, using the fact that $\Gamma_1=0$ and $\Gamma_2$ is positive definite, we obtain
  \begin{equation} \label{eqn:Jacobian_BPDn_p<2}
    \mathbf J_{(x, \mu)} F(x^\diam, \mu^\diam, A^\diam, y^\diam) \, = \, \begin{bmatrix}
     I & 0 & 0 \\
   2 \mu^\diam \Gamma_2 A^T_2 A_1 & I + 2 \mu^\diamond \Gamma_2 A^T_2 A_2 & 2 \Gamma_2 A^T_2 (A^\diamond x^\diamond - y^\diamond) \\ 2 (A^\diamond x^\diamond - y^\diamond)^T A_1 & 2 (A^\diamond x^\diamond - y^\diamond)^T A_2 & 0 \end{bmatrix}.
  \end{equation}
 Since $\Gamma_2$ is positive definite, the lower diagonal block in $\mathbf J_{(x, \mu)} F(x^\diam, \mu^\diam, A^\diam, y^\diam)$ becomes
 \begin{equation} \label{eqn:J_block}
     \begin{bmatrix}
       I + 2 \mu^\diamond \Gamma_2 A^T_2 A_2 & 2 \Gamma_2 A^T_2 (A^\diamond x^\diamond - y^\diamond) \\ 2 (A^\diamond x^\diamond - y^\diamond)^T A_2 & 0 \end{bmatrix} \, = \, \begin{bmatrix} \Gamma_2 & 0 \\ 0 & I \end{bmatrix}  \cdot \underbrace{ \begin{bmatrix}
        \Gamma^{-1}_2 + 2 \mu^\diamond  A^T_2 A_2 & 2 A^T_2 (A^\diamond x^\diamond - y^\diamond) \\ 2 (A^\diamond x^\diamond - y^\diamond)^T A_2 & 0 \end{bmatrix} }_{:=Q}.
 \end{equation}
 Clearly, $\Gamma^{-1}_2 + 2 \mu^\diamond  A^T_2 A_2$ is positive definite. Further,
 since $\mu^\diamond>0$ and $x^\diamond_i \ne 0, \forall \, i \in \Jcal$, we have $A^T_2 (A^\diamond x^\diamond - y^\diamond ) \ne 0$. Hence, by the similar argument for Proposition~\ref{prop:full_support_BPDN_P>2}, we see that the matrix $Q$ is invertible such that $\mathbf J_{(x, \mu)} F(x^\diam, \mu^\diam, A^\diam, y^\diam)$ is invertible. By the implicit function theorem,
 there are local $C^1$ functions $G_1, G_2, H$ such that $x^*=(x^*_{\Jcal^c}, x^*_\Jcal)=(G_1(A, y), G_2(A, y)):=G(A, y)$, $\mu=H(A, y)$, and $F(G(A, y), H(A, y), A, y)=0$ for all $(A, y)$ in a neighborhood of $(A^\diam, y^\diam)$.  By the chain rule, we obtain
 \[
    \mathbf J_{(x, \mu)} F(x^\diam, \mu^\diam, A^\diam, y^\diam) \cdot \begin{bmatrix} \nabla_y G_1(A^\diam, y^\diam) \\ \nabla_y G_2(A^\diam, y^\diam) \\ \nabla_y H(A^\diam, y^\diam) \end{bmatrix} \, = \, - \mathbf J_y F(x^\diam, \mu^\diam, A^\diam, y^\diam) \, = \, \begin{bmatrix} 0 \\  \Gamma_2 \, 2  \mu^\diam A^T_2 \\  2 (A^\diam x^\diam - y^\diam)^T \end{bmatrix}, 
 \]
 where we use the fact that $\Gamma_1=0$.
%
%
 In view of (\ref{eqn:Jacobian_BPDn_p<2}) and the above results, we have $\nabla_y G_1(A^\diam, y^\diam) =0 $, and we deduce via (\ref{eqn:J_block}) that
 \begin{equation} \label{eqn:partial_derive_G2_mu}
  \begin{pmatrix} \nabla_y  G_2(A^\diam, y^\diam) \\ \nabla_y  H(A^\diam, y^\diam)  \end{pmatrix} \, = \, \begin{bmatrix} \Gamma^{-1}_2 + 2 \mu^\diamond A^T_2 A_2 & 2 A^T_2 (A^\diamond x^\diamond - y^\diamond) \\  2 (A^\diamond x^\diamond - y^\diamond)^T A_2 & 0 \end{bmatrix}^{-1} \cdot
  \begin{pmatrix}2 \mu^\diam A^T_2 \\ 2 (A^\diamond x^\diamond - y^\diamond)^T \end{pmatrix},
\end{equation}
where $A^\diamond x^\diamond - y^\diamond \ne 0$ because otherwise $\|A^\diamond x^\diamond - y^\diamond \|^2_2-\varepsilon^2 \ne 0$.

For each $i\in \Jcal^c$, define $q_i(A, y):=a^T_i(y- A \cdot x^*(A, y))$. It follows from $\mbox{sgn}(x^*_i(A, y)) =\mbox{sgn}(q_i(A, y))$ and the previous argument that it suffices to show that $\nabla_y \, q_i(A^\diam, y^\diam) \ne 0$ for each $i \in \Jcal^c$, where $\nabla_y \, q_i(A^\diam, y^\diam) = (a^\diam_i)^T (I-A_2 \cdot \nabla_y \, G_2(A^\diam, y^\diam) )$. Toward this end, we see, by using $(a^\diam_i)^T (A^\diamond x^\diamond - y^\diamond) = 0, \forall \, i \in \mathcal J^c$ and (\ref{eqn:partial_derive_G2_mu}), that for each $i \in \Jcal^c$,
\begin{align*}
\lefteqn{  (a^\diam_i)^T A_2 \nabla_y G_2(A^\diam, y^\diam)  \, = \, \frac{(a^\diam_i)^T}{2\mu^\diamond} \begin{bmatrix} 2\mu^\diamond A^\diamond_2 & 2(A^\diamond x^\diamond - y^\diamond) \end{bmatrix} \cdot \begin{pmatrix} \nabla_y  G_2(A^\diam, y^\diam) \\ \nabla_y  H(A^\diam, y^\diam)  \end{pmatrix} }\\
   & \, = \,
   \frac{(a^\diam_i)^T}{2\mu^\diamond} \begin{bmatrix} 2\mu^\diamond A^\diamond_2 & 2(A^\diamond x^\diamond - y^\diamond) \end{bmatrix} \cdot
   \begin{bmatrix} \Gamma^{-1}_2 + 2 \mu^\diamond A^T_2 A_2 & 2 A^T_2 (A^\diamond x^\diamond - y^\diamond) \\  2 (A^\diamond x^\diamond - y^\diamond)^T A_2 & 0 \end{bmatrix}^{-1} \cdot
  \begin{bmatrix}2 \mu^\diam A^T_2 \\ 2 (A^\diamond x^\diamond - y^\diamond)^T \end{bmatrix}.
\end{align*}
Define
\[
   d \, := \, A^\diamond x^\diamond - y^\diamond \ne 0, \quad C \, := \, \begin{bmatrix} \, \sqrt{2 \mu^\diam} \cdot A_2, & \sqrt{\frac{2}{\mu^\diam}} \cdot d  \, \end{bmatrix}, \quad D \, := \, \begin{bmatrix} \Gamma^{-1}_2 & 0 \\ 0 & - \frac{2}{\mu^\diam} \| d \|^2_2  \end{bmatrix}.
\]
It is easy to verify that
\[
  \begin{bmatrix} \Gamma^{-1}_2 + 2 \mu^\diamond A^T_2 A_2 & 2 A^T_2 (A^\diamond x^\diamond - y^\diamond) \\  2 (A^\diamond x^\diamond - y^\diamond)^T A_2 & 0 \end{bmatrix} \, = \, D + C^T C.
\]
Therefore, we obtain
 \[
   (a^\diam_i)^T \big( I- A_2 \cdot \nabla_y G_2(A^\diam, y^\diam)  \big) \, = \, (a^\diam_i)^T - (a^\diam_i)^T C (D + C^T C)^{-1} C^T.
 \]
 Recall that $\Jcal$ is nonempty such that $A_2$ exists and $A^T_2 d \ne 0$.
 Since $A_2 \Gamma_2 A^T_2$ and $I - \frac{d d^T}{\| d \|^2_2}$ are both positive semi-definite and $N(I - \frac{d d^T}{\| d \|^2_2})=\mbox{span}\{ d \}$, it is easy to see that $N(A_2 \Gamma_2 A^T_2)\cap  N(I - \frac{d d^T}{\| d \|^2_2}) =\{ 0 \}$. Hence, the following matrix is positive definite:
 \[
   I + C D^{-1} C^T \, = \, 2 \mu^\diam A_2 \Gamma_2 A^T_2 + I - \frac{d d^T}{\| d \|^2_2}.
 \]
 By the Sherman-Morrison-Woodbury formula \cite[Section 3.8]{Meyer_book00}, we have
 \[
  C (D + C^T C)^{-1} C^T \, = \, I - (I + C D^{-1} C^T)^{-1}.
 \]
  Consequently, for any $(A^\diam, y^\diam)\in \breve S$, it follows from
 that $a^\diam_i \ne 0, \forall \, i$ that for each $i \in \Jcal^c$,
 \begin{align*}
  \nabla_y \, q_i(A^\diamond, y^\diamond) & \, = \, (a^\diam_i)^T \big(I- A_2 \cdot \nabla_y G_2(A^\diam, y^\diam)  \big) =  (a^\diam_i)^T - (a^\diam_i)^T C (D + C^T C)^{-1} C^T  \\
    & = (a^\diam_i)^T (I + C D^{-1} C^T)^{-1} \, \ne \, 0.
 \end{align*}
By Lemma~\ref{lem:measure_of_zero_set},  the zero set of $x^*_i(A, y)$ has zero measure for each $i=1, \ldots, N$. Therefore, $| \mbox{supp}(x^*(A, y))| = N$ for almost all $(A, y) \in \mathbb R^{m\times N} \times \mathbb R^m$.
\end{proof}

Putting Propositions~\ref{prop:full_support_BPDN_P>2} and \ref{prop:BPDN_p<2} together, we obtain the following result.

\begin{theorem} \label{thm:BPDn_p}
  Let $p>1$ and $N \ge m$. For almost all $(A, y) \in \mathbb R^{m\times N} \times \mathbb R^m$ with $y \ne 0$, if $0<\varepsilon<\| y \|_2$, then the unique optimal solution $x^*_{(A, y)}$ to the BPDN$_p$ (\ref{eqn:BP_denoising01}) satisfies $| \mbox{supp}(x^*_{(A, y)})| = N$.
\end{theorem}

Next, we extend the above result to the optimization problem (\ref{eqn:BP_denoising02}) pertaining to another  version of the generalized basis pursuit denoising under a suitable assumption on $\eta$.
%
%
Since its proof follows an argument similar to that for Theorem~\ref{thm:BPDn_p}, we will be concise on the overlapping parts.

\begin{theorem} \label{thm:BPDn02_p}
  Let $p>1$ and $N \ge m$. For almost all $(A, y) \in \mathbb R^{m\times N} \times \mathbb R^m$ with $y\in R(A)$, if $0<\eta<\min_{Ax = y} \| x\|_p$, then the unique optimal solution $x^*_{(A, y)}$ to (\ref{eqn:BP_denoising02}) satisfies $| \mbox{supp}(x^*_{(A, y)})| = N$.
\end{theorem}

\begin{proof}
 We consider two cases as follows: (i) $p \ge 2$, and (ii) $1<p\le 2$.

 (i) $p\ge 2$. Consider the set $S$ defined in (\ref{eqn:set_S}). Clearly, $A$ has full row rank and $y \in R(A)$ for any $(A, y) \in S$.
%
%
It follows from Proposition~\ref{prop:BPDN2_lower_bound} that  the optimal solution $x^*$ is unique and the associated unique Lagrange multiplier $\mu$ is positive. Define $\wt \mu:=1/\mu>0$. Hence, $(x^*, \wt\mu)$ is a function of $(A, y)$ on $S$ and satisfies the following equation obtained from (\ref{eqn:BPDN02_KKT}):
 \[
 F(x, \wt\mu, A, y) \, := \,
  \begin{bmatrix} g(x_{1}) + \wt \mu \, a^T_1 (A x - y) \\ \vdots \\ g(x_N) +  \wt\mu \, a^T_N (A x - y) \\  f(x) -\eta^p \end{bmatrix} \, = \, 0.
  %
 \]
 Clearly, $F:\mathbb R^N \times \mathbb R \times \mathbb R^{m\times N} \times \mathbb R^m \rightarrow \mathbb R^{N+1}$ is $C^1$ and its Jacobian with respect to $(x, \wt \mu)$ is
 \[
   \mathbf J_{(x, \wt \mu)} F(x, \wt\mu, A, y) \, = \, \begin{bmatrix} M(x, \wt\mu, A) & A^T(A x - y) \\ (\nabla f(x))^T & 0 \end{bmatrix},
 \]
 where $M(x, \mu, A) : =  \Lambda(x) + \wt \mu A^T A$, and the diagonal matrix $\Lambda(x):=\diag(g'(x_1), \ldots, g'(x_N))$ is positive semi-definite.
  For any $(A^\diamond, y^\diamond) \in S$, we use the same notation $x^\diam$, $\wt\mu^\diam$, $\Jcal$, $\Lambda^\diam=\diag(\Lambda_1, \Lambda_2)$ and $A^\diam=[ A_1 \ A_2]$  as before, where $\Lambda_1=0$. Note that in light of $N \ge m$ and the second statement of Proposition~\ref{prop:BPDN2_lower_bound}, we have $|\mbox{supp}(x^\diam)| \ge N-m+1 \ge 1$ such that the index set $\Jcal$ is nonempty. It follows from (\ref{eqn:BPDN02_KKT}) that $\nabla f(x^\diam)+ \wt \mu^\diam \cdot (A^\diam)^T (A^\diam x^\diam-y^\diam)=0$. Further, $\nabla f(x^\diam)\ne 0$ and $(A^\diam)^T (A^\diam x^\diam-y^\diam) \ne 0$. Therefore, using the similar argument as in Proposition~\ref{prop:full_support_BPDN_P>2}, we deduce that $\mathbf J_{(x, \wt \mu)} F(x^\diam, \wt\mu^\diam, A^\diam, y^\diam)$ is invertible for any $(A^\diamond, y^\diamond) \in S$.
By the implicit function theorem, there are local $C^1$ functions $G_1, G_2, H$ such that $x^*=(x^*_{\Jcal^c}, x^*_\Jcal)=(G_1(A, y), G_2(A, y)):=G(A, y)$, $\wt\mu=H(A, y)$, and $F(G(A, y), H(A, y), A, y)=0$ for all $(A, y)$ in a neighborhood of $(A^\diam, y^\diam)$.
  By the chain rule, we have
\[
  \mathbf J_{(x, \wt\mu)} F(x^\diam, \wt\mu^\diam, A^\diam, y^\diam) \cdot \begin{bmatrix} \nabla_y G_1(A^\diam, y^\diam) \\ \nabla_y G_2(A^\diam, y^\diam) \\ \nabla_y H(A^\diam, y^\diam) \end{bmatrix} \, = \, - \mathbf J_{y} F(x^\diam, \wt\mu^\diam, A^\diam, y^\diam) \, = \, \begin{bmatrix} \wt \mu^\diam A^T_1 \\ \wt \mu^\diam A^T_2 \\  0 \end{bmatrix}.
\]
Since $\Lambda_{1}=0$ and $g(x^\diam_i)=0$ for each $i\in \Jcal^c$, we have
\[
     \mathbf J_{(x, \wt\mu)} F(x^\diam, \wt\mu^\diam, A^\diam, y^\diam)  = \begin{bmatrix}  \wt\mu^\diamond A^T_1 A_1 &  \wt \mu^\diam  A^T_1 A_2 &  0 \\  \wt \mu^\diam A^T_2 A_1 & \Lambda_2 +  \wt\mu^\diam  A^T_2 A_2 &  \star \\ 0 & \star & 0 \end{bmatrix}.
\]
Let $P =\begin{bmatrix} P_{11} & P_{12} & P_{13} \\ P_{21} & P_{22} & P_{23} \\ P_{31} & P_{32} & P_{33} \end{bmatrix}$ be the inverse of $\mathbf J_{(x, \wt\mu)} F(x^\diam, \wt\mu^\diam, A^\diam, y^\diam)$.
Hence, we have
$
  P_{11}  \wt \mu^\diamond A^T_1 A_1 + P_{12} \, \wt \mu^\diam A^T_2 A_1  = I_m.
$
Since
$
  \nabla_y G_1(A^\diam, y^\diam) = -[ P_{11} \ P_{12} \ P_{13} ] \cdot \mathbf J_{y} F(x^\diam, \mu^\diam, A^\diam, y^\diam) = P_{11}   \wt \mu^\diamond A^T_1  + P_{12} \, \wt \mu^\diam A^T_2,
$
we have $\nabla_y G_1(A^\diam, y^\diam) \cdot A_1 = I_m$. This shows that each row of $\nabla_y G_1(A^\diam, y^\diam)$ is nonzero. Hence, following the exactly same argument as in Proposition~\ref{prop:full_support_BPDN_P>2}, we deduce that $| \mbox{supp}(x^*(A, y))| = N$ for almost all $(A, y) \in \mathbb R^{m\times N} \times \mathbb R^m$.

(ii) $1<p\le 2$. Let $\grave S$ be the set of $(A, y)\in \mathbb R^{m\times N}\times \mathbb R^m$ with $N \ge m$ such that $y \ne 0$, each column of $A$ is nonzero, and $A$ has full row rank. Hence, $y\in R(A)$ for any $(A, y) \in \grave S$.
%
%
By Proposition~\ref{prop:BPDN2_lower_bound}, we see that (\ref{eqn:BP_denoising02}) attains a unique optimal solution $x^*$  and a unique Lagrange multiplier $\mu>0$ for any $(A, y)\in \grave S$. Define $\wt \mu:=1/\mu>0$. Hence, $(x^*, \wt\mu)$, as a function of $(A, y)$ on $\grave S$, satisfies the following equation:
\[
 F(x, \mu, A, y) \, := \,  \begin{bmatrix} x_{1} + h(\wt\mu a^T_1 (A x - y)) \\ \vdots \\ x_N + h(\wt\mu a^T_N (A x - y))\\ f(x) - \eta^p  \end{bmatrix} \, = \, 0.
 \]
 Here $F:\mathbb R^N \times \mathbb R \times \mathbb R^{m\times N} \times \mathbb R^m \rightarrow \mathbb R^{N+1}$ is $C^1$ and its Jacobian with respect to $(x, \wt \mu)$ is given by
 \[
   \mathbf J_{(x, \wt \mu)} F(x, \wt \mu, A, y) \, = \, \begin{bmatrix} V(x, \wt \mu, A, y) &  \Gamma(x, \wt \mu, A, y) A^T (A x - y) \\ (\nabla f(x))^T & 0 \end{bmatrix} \in \mathbb R^{(N+1)\times (N+1)},
 \]
 where  $\Gamma(x, \wt\mu, A, y):=\mbox{diag}\big( h'(\wt\mu a^T_1 (A x - y)), \ldots, h'(\wt\mu a^T_N (A x - y))  \big) \in \mathbb R^{N\times N}$, and $V(x, \wt \mu, A, y) := I + \Gamma(x, \wt \mu, A, y) \, \wt \mu A^T A$. Using the same notation introduced in Proposition~\ref{prop:BPDN_p<2}, we deduce that for any $(A^\diam, y^\diam)\in \grave S$,
 \begin{equation*} 
    \mathbf J_{(x, \wt\mu)} F(x^\diam, \wt \mu^\diam, A^\diam, y^\diam) \, = \, \begin{bmatrix}
     I & 0 & 0 \\
   \star & I + \wt \mu^\diamond \Gamma_2 A^T_2 A_2 &  \Gamma_2 A^T_2 (A^\diamond x^\diamond - y^\diamond) \\ \star & v^T & 0 \end{bmatrix},
  \end{equation*}
 where the column vector $v:=(g(x^\diam_i))_{i\in \mathcal J}$. Note that the index set $\mathcal J$ is nonempty since $y \ne 0$ and $A$ has full row rank such that $A^T y \ne 0$. In view of (\ref{eqn:BPDN02_KKT}), we have $v = -\wt\mu^\diam A^T_2 (A^\diamond x^\diamond - y^\diamond)$, where $\wt\mu^\diam>0$. This result, along with the similar argument for (\ref{eqn:J_block}), shows that $\mathbf J_{(x, \wt \mu)} F(x^\diam, \wt \mu^\diam, A^\diam, y^\diam)$ is invertible.
 Therefore, there are local $C^1$ functions $G_1, G_2, H$ such that $x^*=(x^*_{\Jcal^c}, x^*_\Jcal)=(G_1(A, y), G_2(A, y)):=G(A, y)$, $\wt\mu=H(A, y)$, and $F(G(A, y), H(A, y), A, y)=0$ for all $(A, y)$ in a neighborhood of $(A^\diam, y^\diam)$.  Moreover,
 \[
    \mathbf J_{(x, \wt \mu)} F(x^\diam, \wt \mu^\diam, A^\diam, y^\diam) \cdot \begin{bmatrix} \nabla_y G_1(A^\diam, y^\diam) \\ \nabla_y G_2(A^\diam, y^\diam) \\ \nabla_y H(A^\diam, y^\diam) \end{bmatrix} \, = \, - \mathbf J_y F(x^\diam,  \wt \mu^\diam, A^\diam, y^\diam) \, = \, \begin{bmatrix} 0 \\  \Gamma_2 \, \wt  \mu^\diam A^T_2 \\  0 \end{bmatrix}, 
 \]
 where the fact that $\Gamma_1=0$ is used. Therefore,  we have $\nabla_y G_1(A^\diam, y^\diam) =0$, and
 \begin{equation*} 
  \begin{pmatrix} \nabla_y  G_2(A^\diam, y^\diam) \\ \nabla_y  H(A^\diam, y^\diam)  \end{pmatrix} \,  = \, \begin{bmatrix} \Gamma^{-1}_2 + \wt \mu^\diamond A^T_2 A_2 &  A^T_2 (A^\diamond x^\diamond - y^\diamond) \\  (A^\diamond x^\diamond - y^\diamond)^T A_2 & 0 \end{bmatrix}^{-1} \cdot
  \begin{pmatrix} \wt \mu^\diam A^T_2 \\ 0 \end{pmatrix},
\end{equation*}
 In what follows, define $b:=A^T_2(A^\diamond x^\diamond - y^\diamond) \ne 0$ and $M:=\Gamma^{-1}_2 + \wt \mu^\diamond A^T_2 A_2$, which is positive definite.

For each $i\in \Jcal^c$, define $q_i(A, y):=(a^\diam_i)^T (y-A \cdot x^*(A, y))$. It suffices to show that $\nabla_y \, q_i(A^\diam, y^\diam) \ne 0$ for each $i \in \Jcal^c$, where $\nabla_y \, q_i(A^\diam, y^\diam) = (a^\diam_i)^T (I-A_2 \cdot \nabla_y \, G_2(A^\diam, y^\diam) )$. Direct calculations show that
\begin{align*}
  I - A_2 \cdot \nabla_y \, G_2(A^\diam, y^\diam)  & = I - [ A_2 \ \, 0 ] \cdot \begin{bmatrix} M & b \\ b^T & 0\end{bmatrix}^{-1}  \cdot \begin{pmatrix} \wt \mu^\diam A^T_2 \\ 0 \end{pmatrix}  = I - A_2 M^{-1}\Big(I - \frac{b b^T M^{-1}}{b^T M^{-1} b} \Big) \wt\mu^\diam A^T_2  \\
  & = I- A_2 M^{-1} \wt\mu^\diam A^T_2  + \wt \mu^\diam \frac{A_2 M^{-1} b b^T M^{-1}A^T_2}{b^T M^{-1}  b}.
\end{align*}
By the definition of $M$ and the Sherman-Morrison-Woodbury formula \cite[Section 3.8]{Meyer_book00}, we have $I- A_2 M^{-1} \wt\mu^\diam A^T_2  = (I+ \wt \mu^\diam A_2 \Gamma_2 A^T_2)^{-1}$, which is positive definite. Hence, $I-A_2 \cdot \nabla_y \, G_2(A^\diam, y^\diam)$ is positive definite and thus invertible. Since $a^\diam_i \ne 0$, we have $\nabla_y \, q_i(A^\diam, y^\diam) \ne 0$ for each $i \in \Jcal^c$. Consequently, $| \mbox{supp}(x^*(A, y))| = N$ for almost all $(A, y) \in \mathbb R^{m\times N} \times \mathbb R^m$.
\end{proof}

%
\section{Extension and Comparison} \label{sect:comparison}


In this section, we extend the least sparsity results to constrained measurement vectors, and then compare  the least sparsity results for $p>1$ with those from $p$-norm based optimization for $0<p\le 1$.

%
\subsection{Extension to Constrained Measurement Vectors}

In the previous sections, we consider general measurement vectors in $\mathbb R^m$. However, in many applications such as compressed sensing, a measurement vector $y$ is restricted to a proper subspace of $R(A)$, to which the results in Section~\ref{sect:least_sparsity} are not applicable since this subspace may have dimension less than $m$ so that it has zero measure in $\mathbb R^m$. In what follows, we extend the least sparsity results in Section~\ref{sect:least_sparsity} to this scenario. For simplicity, we consider the generalized basis pursuit BP$_p$ (\ref{eqn:BP_p_norm}) with $p>1$ only, although its result can be extended to the other generalized optimization problems, e.g., the BPDN$_p$, RR$_p$, and EN$_p$; see Remark~\ref{remark:RR_p} for an example.

\begin{theorem} \label{thm:extension_BP_p}
 Let $p>1$, $N \ge 2m-1$, and  $\Ical \subseteq\{1, \ldots, N\}$ be a nonempty index set. Then there exists a set $S_A \subset \mathbb R^{m\times N}$ whose complement has zero measure such that for each fixed $A \in S_A$,
 the unique optimal solution $x^*$ to the BP$_p$ (\ref{eqn:BP_p_norm}) satisfies $|\mbox{supp}(x^*)| = N$ for almost all $y\in R(A_{\bullet \Ical})$.
\end{theorem}

\begin{proof}
 For the given $p>1$, $N, m\in \mathbb N$ with $N \ge 2m -1$ and the index set $\Ical$, we consider two cases: (i) $|\Ical| \ge m$; and (ii) $|\Ical| < m$. For the first case, let $S_A$ be the set of all $A\in \mathbb R^{m\times N}$ such that any $m\times m$ submatrix of $A$ is invertible. Clearly, the complement of $S_A$ has zero measure in the space $\mathbb R^{m\times N}$. Further, since $|\Ical| \ge m$, $R(A_{\bullet \Ical})=\mathbb R^m$ for any $A \in S_A$. Hence, it follows from Corollary~\ref{coro:BP_p_A_fixed} that the proposition holds.
%
%

 We then consider the second case where $|\Ical| < m$. In this case, let $\wt S_A$ be the set of all $A\in \mathbb R^{m\times N}$ satisfying the following condition: for any index set $\Jcal$ with $|\Jcal|=|\Ical|:=r$, the $r\times r$ matrix $(A_{\bullet \Ical})^T \cdot A_{\bullet \Jcal}$ is invertible. Note that for any index set $\Jcal$ with $|\Jcal|=r$, $\det( (A_{\bullet \Ical})^T A_{\bullet \Jcal} )=0$ gives rise to a polynomial equation of the elements of $A$. Hence, we deduce that the complement of $\wt S_A$ has zero measure in  $\mathbb R^{m\times N}$. Further, for any $A \in \wt S_A$, the columns of $A_{\bullet \Ical}$ must be linearly independent. Therefore, for any $y \in R(A_{\bullet \Ical})$, there exists a unique $z_y \in \mathbb R^r$ such that $A_{\bullet \Ical} \cdot z_y = y$. This shows that the constraint $Ax = y$ in  the BP$_p$ (\ref{eqn:BP_p_norm}) can be equivalently written as
 \[
     \big[ (A_{\bullet \Ical})^T  A_{\bullet \Ical} \big]^{-1} \cdot (A_{\bullet \Ical})^T \cdot A x \, = \, z_y.
 \]
 Define the $r\times N$ matrix $\wt A := [ (A_{\bullet \Ical})^T  A_{\bullet \Ical} ]^{-1}  (A_{\bullet \Ical})^T  A$ for each $A\in \wt S_A$. It follows from the property of $A \in \wt S_A$ that any $r\times r$ submatrix of  $\wt A$ is invertible. Hence, for any $A\in \wt S_A$ and any $y \in R(A_{\bullet \Ical})$,  the original BP$_p$ (\ref{eqn:BP_p_norm}) is converted to the following equivalent optimization problem: for some $z \in \mathbb R^r$,
 \begin{equation} \label{eqn:BP_p_converted}
   \underset{x \in \mathbb R^N}{\text{min}} \ \|x\|_p  \quad \text{subject to} \quad \wt A x = z.
 \end{equation}
 For a fixed $\wt A$ obtained from a given $A \in \wt S_A$, by applying  Corollary~\ref{coro:BP_p_A_fixed} to (\ref{eqn:BP_p_converted}), we deduce that $| \mbox{supp}(x^*(z))| = N$ for almost all $z \in \mathbb R^r$.
 Since $A_{\bullet \Ical}$ has full column rank, the same conclusion holds for almost all $y\in R(A_{\bullet \Ical})$.
\end{proof}

The following corollary can be easily established with the aid of Theorem~\ref{thm:extension_BP_p} and the extension of Proposition~\ref{prop:sparsity_lower_bound} to (\ref{eqn:BP_p_converted}); its proof is thus omitted.
%
%

\begin{corollary} \label{coro:BP_p_restriction}
Let $p >1$, $N \ge 2m-1$, and $s\in \{1, \ldots, N\}$. The following hold:
\begin{itemize}
  \item [(i)] There exists a set $S_A \subset \mathbb R^{m\times N}$ whose complement has zero measure such that for each fixed $A \in S_A$ and any index set $\Ical$ with $|\Ical| \le s$,   the unique optimal solution $x^*$ to the BP$_p$ (\ref{eqn:BP_p_norm}) satisfies $|\mbox{supp}(x^*)| \ge N-m+1$ for  any nonzero $y\in R(A_{\bullet \Ical})$;
  \item [(ii)] There exists a set $\wh S_A \subset \mathbb R^{m\times N}$ whose complement has zero measure such that for each fixed $A \in \wh S_A$ and any index set $\Ical$ with $|\Ical| \le s$,   the unique optimal solution $x^*$ to the BP$_p$ (\ref{eqn:BP_p_norm}) satisfies $|\mbox{supp}(x^*)| = N$ for almost all $y\in R(A_{\bullet \Ical})$.
\end{itemize}
\end{corollary}

%
\subsection{Comparison with $p$-norm based Optimization with $0<p\le 1$}

For a given sparsity level $s$ with $1\le s \le N$ (especially $s \ll N$), we call a vector $x \in \mathbb R^N$ $s$-sparse if $|\mbox{supp}(x)| \le s$. Furthermore, we call a measurement vector $y$ generated by an $s$-sparse vector if there is an $s$-sparse vector such that $y= A x$. Using these terminologies, we see that  Corollary~\ref{coro:BP_p_restriction} states that when $p>1$, for almost all $A \in \mathbb R^{m\times N}$ with $N \gg \max(m, s)$ and almost all $y$ generated by $s$-sparse vectors, the optimal solution $x^*$ to the BP$_p$ (\ref{eqn:BP_p_norm}) is far from sparse, i.e., $|\mbox{supp}(x^*)| \gg s$.  Equivalently, it means that when $p>1$,  the BP$_p$ (\ref{eqn:BP_p_norm}) {\em might} recover a sparse vector $x$ from $y=A x$ only for a  set of $A$'s of zero measure in $\mathbb R^{m\times N}$, no matter how large $N$ and $m$ are. Moreover, an arbitrarily small perturbation to a measurement matrix $A$ in this zero measure set will lead to a least sparse solution. This shows the extremely weak robustness of the BP$_p$ (\ref{eqn:BP_p_norm}) with $p>1$ in term of solution sparsity.


For comparison, it is interesting to ask what happens to the BP$_p$ (\ref{eqn:BP_p_norm}) when $0<p \le 1$.
We show below that when $0<p\le 1$, there exists a non-zero measure set of $A$'s such that the BP$_p$ (\ref{eqn:BP_p_norm}) recovers any sparse vector $x$ from $y=A x$. This result also demonstrates the strong robustness of the BP$_p$ (\ref{eqn:BP_p_norm}) for $0<p\le 1$.
Toward this end, recall that an $m \times N$ matrix $A$ satisfies the {\em restricted isometry property} (RIP) of order $k$ if there is a constant $\delta_k \in (0, 1)$ such that $(1-\delta_k) \| x \|^2_2 \le \| A x \|^2_2 \le (1+\delta_k) \| x \|^2_2$ for all  $k$-sparse vectors $x \in \mathbb R^N$.

\begin{proposition}
  Fix $p\in (0, 1]$, $\gamma \in (0, 1)$ and $s \in \mathbb N$. Suppose $m=\lceil \gamma N \rceil$. Then  for all $N$ sufficiently large, there exists an open set $U_A \subset \mathbb R^{m\times N}$ such that for any $A \in U_A$ and any index set $\Ical$ with $|\Ical| \le s$, the optimal solution $x^*$ to the BP$_p$ (\ref{eqn:BP_p_norm}) satisfies $|\mbox{supp}(x^*)| = |\Ical|$ for all $y \in R(A_{\bullet \Ical})$.
\end{proposition}

\begin{proof}
 Consider $p=1$ first, which corresponds to the $\ell_1$-optimization based basis pursuit \cite{FoucartRauhut_book2013}. For the given constants $\gamma\in (0, 1)$, the sparsity level $s$, and $\delta_{3s} \in (0, 1/3)$, it is known via random matrix argument that  for all $N$ sufficiently large with $s\ll m$, there exists a matrix $A^\diam \in \mathbb R^{m\times N}$ which satisfies the RIP of order $3s$ with constant $\delta_{3s}(A^\diam) < 1/3$ \cite{Baraniuk_CA08, Bryan_SIREW13}. Hence, the BP$_p$ (\ref{eqn:BP_p_norm}) recovers any $s$-sparse vector $x$ exactly from $y = A^\diam x$ \cite[Theorem 6.9]{FoucartRauhut_book2013} or \cite{Bryan_SIREW13}. Furthermore, in view of $\big | \| A x \|_2- \| A^\diam x \|_2 \big| \le \| A - A^\diam\|_2 \cdot \|x \|_2$ for any $A$ and $x$, we see that
 %
 %
 there exists $\eta>0$ such that $\delta_{3s}(A) < 1/3$ for all $A$'s with $\| A - A^\diam \|_2 < \eta$. Let the open set $U_A:=\{ A \in \mathbb R^{m\times N} \, | \,  \| A - A^\diam \|_2 < \eta \}$.
 This shows that for any $A\in U_A$, every $s$-sparse vector $x$ can be recovered from $y=A x$ via the BP$_p$ (\ref{eqn:BP_p_norm}).
%
%
 Finally, when $0<p<1$, it follows from \cite[Theorem 4.10]{FoucartRauhut_book2013} that  for any $A \in U_A$, the BP$_p$ (\ref{eqn:BP_p_norm}) recovers any $s$-sparse vector $x$ from $y = A x$. This completes the proof.
\end{proof}

\begin{remark} \label{remark:RR_p} \rm
 Consider the ridge regression RR$_p$ (\ref{eqn:ridge_reg}), and we compare the sparsity property for $p>1$ with that for $0<p<1$. It is shown in \cite[Theorem 2.1(2)]{ChenFY_SSC10} that when $0<p<1$, for any $A \in \mathbb R^{m\times N}$, $y \in \mathbb R^m$, and $\lambda>0$, any (local/global) optimal solution $x^*$ to the RR$_p$ (\ref{eqn:ridge_reg}) satisfies $|\mbox{supp}(x^*)| \le m$. In contrast, Theorem~\ref{thm:RR_p} shows that when $p>1$, an optimal solution $x^*$ to (\ref{eqn:ridge_reg}) has full support, i.e., $|\mbox{supp}(x^*)| = N$, for almost all $A$ and $y$.
\end{remark}

%
\section{Conclusions} \label{sect:conclusion}

This paper provides an in-depth study of sparse properties of a wide range of  $p$-norm based optimization problems with $p>1$ generalized from sparse optimization and other related areas. By applying optimization and matrix analysis techniques, we show that optimal solutions to these generalized problems are the least sparse for almost all measurement matrices and measurement vectors. We also compare these problems with those when $0<p\le 1$. The results of this paper not only give a formal justification of the usage of $\ell_p$-optimization with $0<p\le 1$ for sparse optimization but they also offer a quantitative characterization of the adverse sparsity properties of $\ell_p$-optimization with $p>1$. These results will shed light on analysis and computation of general $p$-norm based optimization problems. Future research includes extensions to matrix norm based optimization problems.

%
\section{Appendix} \label{sect:appendix}

In what follows, we show that the function $\| \cdot \|^p_p$ with $p>1$ is strictly convex.

\begin{proof}
%
%
Let $p>1$. By the Minkowski inequality, we have $\| x + y \|_p \le \| x \|_p + \| y \|_p$ for all $x, y\in \mathbb R^N$, and the equality holds if and only if $y= \mu x$ for $\mu \ge 0$. For any $x, y \in \mathbb R^N$ with $x \ne y$ and any $\lambda\in (0, 1)$, consider two cases (i) $y =\mu x$ for some $\mu \ge 0$ with $\mu \ne 1$; (ii) otherwise. For case (i), $\| \lambda x + (1-\lambda) y \|^p_p = \| \lambda x + (1-\lambda) \mu x \|^p_p = [1\cdot \lambda +\mu \cdot (1-\lambda)  ]^p \cdot \| x \|^p_p < [\lambda + \mu^p (1-\lambda)] \| x\|^p_p = \lambda \| x \|^p_p +(1-\lambda) \| y \|^p_p$, where we use the fact that $|x|^p$ is strictly convex on $\mathbb R_+$. For case (ii), we have $\| \lambda x + (1-\lambda) y \|^p_p < \big(\lambda  \| x\|_p + (1-\lambda) \| y \|_p \big)^p \le \lambda \| x \|^p_p + (1-\lambda) \| y \|^p_p$. This shows that $\| \cdot \|^p_p$ is strictly convex.
\end{proof}

%


\begin{thebibliography}{99}


{\small


%
%
%

\bibitem{Baraniuk_CA08}
{\sc R. Baraniuk, M. Davenport, R. DeVore, and M. Wakin}. A simple proof of the restricted
isometry property for random matrices. {\sl Constructive Approximation}, Vol. 28, pp. 253--263, 2008.

\bibitem{BeckE_SIOPT13}
{\sc A. Beck and Y.C. Eldar}. Sparsity constrained nonlinear optimization: Optimality conditions and algorithms. {\sl SIAM Journal on Optimization}, Vol. 23(3), pp. 1480--1509, 2013.

\bibitem{Bryan_SIREW13}
{\sc K. Bryan and T. Leise}. Making do more with less: An introduction to compressed sensing. {\sl SIAM Review}, Vol. 55(3), pp. 547--566, 2013.

\bibitem{BKanzowS_SIOP16}
{\sc O.P. Burdakov, C. Kanzow, and A. Schwartz}.  Mathematical programs with cardinality constraints: Reformulation by complementarity-type conditions and a regularization method. {\sl SIAM Journal on Optimization}, Vol. 26(1), pp. 397--425, 2016.


\bibitem{CandesRT_TIT06}
{\sc E.J. Candes, J. Romberg, and T. Tao}.  Robust uncertainty principles: Exact signal reconstruction from highly incomplete frequency information. {\sl IEEE Transactions on Information Theory}, Vol. 52(2), pp. 489--509, 2006.


\bibitem{CandesW_ISP08}
{\sc E.J. Candes and M.B. Wakin}. An introduction to compressive sampling. {\sl IEEE Signal Processing Magazine}, Vol. 25(2), pp. 21--30, 2008.


\bibitem{ChenDS_SIAMREV01}
{\sc S. Chen, D.L. Donoho, and M.A. Saunders}. Atomic decomposition by basis pursuit. {\sl SIAM Review}, Vol. 43(1), pp. 129--159, 2001.

\bibitem{ChenFY_SSC10}
{\sc X. Chen, F. Xu, and Y. Ye}. Lower bound theory of nonzero entries in solutions of $\ell_2-\ell_p$ minimization. {\sl SIAM Journal on Scientific Computing}, Vol. 32(5), pp. 2832--2852, 2010.

\bibitem{ElChamie_ACC14}
{\sc M. El Chamie and G. Neglia}. Newton's method for constrained norm minimization and its application to weighted graph problems. 2014 American Control Conference, pp. 2983--2988, Portland, OR, 2014.

\bibitem{Donoho_TIT06}
{\sc D. Donoho}. Compressed sensing. {\sl IEEE Trans. Information Theory}, Vol. 52, pp. 1289--1306, 2006.

\bibitem{Donoho_CPAM06}
{\sc D. Donoho}. For most large underdetermined systems of linear equations, the minimal $\ell_1$-norm
solution is also the sparsest solution. {\sl Comm. Pure Appl. Math.}, Vol. 59, pp. 797--829, 2006.

\bibitem{FengMPSW_Techreport13}
{\sc M. Feng, J.E. Mitchell, J.-S. Pang, X. Shen, and A. Wachter}. Complementarity formulations of $L_0$-norm optimization problems. Industrial Engineering and Management Sciences. Technical Report. Northwestern University, Evanston, IL, 2013.


\bibitem{FoucartRauhut_book2013}
{\sc S. Foucart and H. Rauhut}. {\sl A Mathematical Introduction to Compressive Sensing}. Springer, New York, 2013.

\bibitem{FourcartLai_ACHA09}
{\sc S. Fourcar and M.-J. Lai}. Sparse solutions of under-determined linear systems via $\ell_q$ minimization for $0<q<1$. {\sl Applied and Computational Harmonic Analysis}, Vol. 26, pp. 395--407, 2009.

\bibitem{GeJYe_MP11}
{\sc D. Ge, X. Jiang, and Y. Ye}.  A note on the complexity of $L_p$ minimization. {\sl Mathematical Programming}, Vol. 129(2), pp. 285--299, 2011.

\bibitem{HastieTF_book09}
{\sc T. Hastie, R. Tibshirani, and J. Friedman}. Overview of Supervised Learning. In {\sl The Elements of Statistical Learning}, pp. 9--41, Springer, New York, 2009.

\bibitem{HoerlK_Techno70}
{\sc A.E. Hoerl and R.W. Kennard}.  Ridge regression: Biased estimation for nonorthogonal problems. {\sl Technometrics}, 12(1), pp. 55--67, 1970.

\bibitem{LaiWang_SIOPT11}
{\sc M.-J. Lai and J. Wang}. An unconstrained $\ell_q$ minimization with $0< q \leq 1$ for sparse solution of underdetermined linear systems. {\sl SIAM Journal on Optimization}, Vol. 21(1), pp. 82--101, 2011.

\bibitem{Lefkimiatis_TIP13}
{\sc S. Lefkimmiatis and M. Unser}. Poisson image reconstruction with Hessian
Schatten-norm regularization. {\sl IEEE Transactions on Image Processing}, Vol. 22(11), pp. 4314--4327, 2013.

\bibitem{Meyer_book00}
{\sc C.D. Meyer}. {\sl Matrix Analysis and Applied Linear Algebra}.  SIAM Press, 2nd Edition, 2000.

\bibitem{Munkres_book00}
{\sc J.R. Munkres}. {\sl Topology}. 2nd Edition, Prentice-Hall, Englewood Cliffs, NJ, 2000.


\bibitem{Royden_book88}
{\sc H.L. Royden}. {\sl Real Analysis}. 3rd Edition, Prentice-Hall, Englewood Cliffs, NJ, 1988.

\bibitem{Terlaky_EJOR85}
{\sc T. Terlaky}. On $\ell_p$ programming. {\sl European Journal of Operational Research}, Vol. 22(1), pp. 70--100, 1985.

\bibitem{RTibshirani_EJS13}
{\sc R. Tibshirani}. The Lasso problem and uniqueness. {\sl Electronic Journal of Statistics}, Vol. 7, pp. 1456--1490, 2013.

\bibitem{RTibshirani_JRSS96}
{\sc R. Tibshirani}. Regression shrinkage and selection via the Lasso. {\sl Journal of the Royal Statistical Society: Series B}, Vol. 58(1), pp. 267--288, 1996.


\bibitem{VanDenBergF_SSC08}
{\sc E. Van Den Berg and M.P. Friedlander}. Probing the Pareto frontier for basis pursuit solutions. {\sl SIAM Journal on Scientific Computing}, Vol. 31(2), pp. 890--912, 2008.

\bibitem{WrightNF_TSP09}
{\sc S.J. Wright, R.D. Nowak, and M.A. Figueiredo}. Sparse reconstruction by separable approximation. {\sl IEEE Transactions on Signal Processing}, Vol. 57(7), pp. 2479--2493, 2009.

\bibitem{Zhou_JRSS05}
{\sc H. Zou and T. Hastie}. Regularization and variable selection via the elastic net. {\sl Journal of
the Royal Statistical Society: Series B}, Vol. 67(2), pp. 301--320, 2005.

%
%


%
%



}

\end{thebibliography}
\end{document}